\documentclass[reqno]{amsart}
\usepackage{amsmath}
\usepackage[english]{babel}
\usepackage[utf8]{inputenc}
\usepackage{amsfonts}
\usepackage{amsthm}
\usepackage{color}
\usepackage{graphicx}
\usepackage{dsfont}

\newtheorem{thm}{Theorem}[subsection]

\newtheorem{prop}[thm]{Proposition}
\newtheorem{cor}[thm]{Corollary}
\newtheorem{lem}[thm]{Lemma}
\newtheorem{rmk}[thm]{Remark}
\newtheorem{exmp}{Example}[subsection]
\numberwithin{equation}{section}

\DeclareMathOperator{\E}{\mathds{E}}
\DeclareMathOperator{\N}{\mathbb{N}}
\DeclareMathOperator{\R}{\mathbb{R}}

\DeclareMathOperator{\cL}{\mathcal{L}}
\DeclareMathOperator{\cB}{\mathcal{B}}

\DeclareMathOperator{\fF}{\mathfrak{F}}

\DeclareMathOperator{\fG}{\mathfrak{G}}
\DeclareMathOperator{\fT}{\mathfrak{T}}

\DeclareMathOperator{\bP}{\mathds{P}}

\DeclareMathOperator{\Dom}{Dom}

\newcommand{\der}[2]{\frac{d #1}{d #2}}

\def \l { \left( }
\def \r {\right) }

\newcommand{\bt}[1]{{\textcolor{black}{#1}}}
\newcommand{\bbt}[1]{{\textcolor{black}{#1}}}

\newcommand{\mg}[1]{{\textcolor{black}{#1}}}

\begin{document}
\title{On the Exit Time from Open Sets of Some Semi-Markov Processes}
\author{Giacomo Ascione}
\address{}
\author{Enrica Pirozzi}
\address{}
\author{Bruno Toaldo}
\address{}
\begin{abstract}
\bt{In this paper we characterize the distribution of the first exit time from an arbitrary open set for a class of semi-Markov processes obtained as time-changed Markov processes. We estimate the asymptotic behaviour of the survival function (for large $t$) and of the distribution function (for small $t$) and we provide some conditions for absolute continuity. We have been inspired by a problem of neurophyshiology and our results are particularly usefull in this field, precisely for the so-called Leacky Integrate-and-Fire (LIF) models: the use of semi-Markov processes in these models appear to be realistic under several aspects, e.g., it makes the intertimes between spikes a r.v. with infinite expectation, which is a desiderable property. Hence, after the theoretical part, we provide a LIF model based on semi-Markov processes.}
\end{abstract}
\maketitle
\tableofcontents
\section{Introduction}
\bt{This paper deals with the problem of studying the distribution of the exit time from an arbitrary open set for a class of semi-Markov processes, constructed as time-changed Markov processes. More precisely, let $M(t)$, $t \geq 0$, be a Markov process and $\sigma(t)$ an independent stricly increasing L\'evy process. Let $L(t)$ be the inverse of $\sigma(t)$ and define $X(t)=M(L(t))$. In recent years this class of processes have attracted the interests of many mathematicians, because of their connection with fractional type equations and since they are very popular in applications (see \cite{meerschaert2011stochastic} for a review) in particular in the field of anomalous diffusive phenomena (e.g. \cite{Metzler}) and many others (see \cite{hairer} for some recent developments). In this paper we consider the following problem. Let $\mathfrak{T}$ be the exit time from an open set of $X$. We study the behaviour as $t \to \infty$ of $P \l \mathfrak{T}>t \r$ and as $t \to 0$ of $P ( \mathfrak{T} \leq t)$. Beside its natural interest as a theoretical question, this problem is inspired by neurophysiology investgations and it turns out that our results are particularly useful in this field, as follows. The stochastic Leaky Integrate-and-Fire models for the membrane potential of a neuron are one of the most popular way to model such dynamics (e.g. \cite{greenwood2016stochastic, Lev15}). However the classical processes used to describe the membrane potential \cite{greenwood2016stochastic, lansky, Lev15, sacric} are such that the first passage time \mg{through} the threshold, upon which the neuron fire, is a r.v. with finite expectation. This is in contrast with the observed behaviour (see, for instance, \cite{gerstein1964random}) since the distribution of the intervals between spikes appears to be heavy tailed. Further, phenomenological evidences such as high variability in the neuronal response to stimulations and the adaptation phenomenon, cannot be explained by models based on Markovian processes, but the introduction of memory seems to be a suitable and powerful tool for modeling such dynamics (see again, for instance, \cite{Lev15} and references therein). Hence, we propose in Section \ref{secmodel} a model based on semi-Markov processes, constructed as above, leading to distributions whose survival function has a $\alpha$-power law decay, $\alpha \in (0,1)$.}

\bt{Therefore, we first accomplish to the theoretical task of characterizing the distribution of the first exit time from an open set for the considered class of semi-Markov processes. For example, it turns out that the behaviour of the tail is a $\alpha$-power law, $\alpha \in (0,1)$, in case $\sigma$ is an
 $\alpha$-stable subordinator and if the function $s \mapsto \mathds{E} \left[ 1-e^{-sT} \right]$, where $T$ is the first exit time of the original Markov process, is regularly varying.
Then, we focus on the situation in which the original Markov process is a Gauss-Markov process, since this kind of processes are usually adopted in LIF models, and we show when they satisfies the condition needed to use our estimates. It turns out that this can be done by means of Doob transformation Theorem.}

\bt{Another feature of our model is that it is a reparametrization of the original one (before the time-change). For example: suppose that the model is obtained by the time-change of a Markov process such that $T$ is its exit time from the open set and $\mathds{E}T=C<\infty$. Suppose that the time-changed process $X$ is obtained with the inverse of an $\alpha$-stable subordinator. Then the tail behaviour of the exit time is $Ct^{-\alpha}/\Gamma (1-\alpha)$ and the parameters $C$ and $\alpha$ can be observed directly by observing the spikes. To highlight this and as a confirmation of our results, in the last sections we provide a method to simulate our processes.}

\section{\bt{The exit time}}\label{sec:main}
\bt{In this section we} study the asymptotic behaviour of the distribution functions of the first passage times of semi-Markov processes obtained by means of a time change from a Markov process.
\subsection{\bt{Construction of the process and general assumptions}}
Let us consider a Markov process $M=\{M(y), \ y \ge 0\}$ with state space $(\Sigma, \mathfrak{G})$, conditional probability laws $(\bP^x)_{x \in \Sigma}$ and infinitesimal generator $G_M$. Let us consider also a subordinator $\sigma=\{\sigma(y), \ y \ge 0\}$ independent on $M$, that is to say a non-decreasing L\'{e}vy process. In particular $\sigma$ has state space $([0,+\infty),\cB[0,+\infty))$ and
\begin{equation}\label{eq:lapexp}
\E^x[e^{-\lambda \sigma(y)}]=e^{-yf(\lambda)}
\end{equation}
where $f(\lambda)$ is a Bernstein function
\begin{equation}\label{eq:Bernfun}
f(\lambda)=\int_0^{+\infty}(1-e^{-\lambda s})\nu(ds).
\end{equation}
\bbt{The measure $\nu(\cdot)$ is} the L\'{e}vy measure of $\sigma$ \bbt{and must} fulfills the integrability condition
\begin{equation}\label{eq:integcond}
\int_0^{+\infty}(s \wedge 1)\nu(ds)<\infty.
\end{equation}
In what follows we will always assume that $\nu(0,+\infty)=+\infty$ to assure that the subordinator is a strictly increasing process, \bt{even if we will always assume that there is no drift}. Let us also define the time-changed process $X^f=\{X^f(t), t \ge 0\}$ as $X^f(t)=(M \circ L)(t)$ where
\begin{equation}\label{eq:invsub}
L(t)=\inf\{y \ge 0 : \, \sigma(y)>t\}
\end{equation}
that is called inverse subordinator since it is the right-continuous inverse of $\sigma$. It is known in \cite{baeumer2001stochastic} that the process $X^f$ is governed by a time-fractional equation when $f(\lambda)=\lambda^\alpha$. Hence it is such that the function $q(x,t):=\E^x[u(X^f(t))]$ satisfies
\begin{equation}\label{eq:fraceqforXf}
\partial_t^\alpha q=G_Mq, \quad q(x,0)=u(x) \in \Dom(G_M),
\end{equation}
where $G_M$ is the generator of $M$ and $\partial_t^\alpha$ is the \bt{fractional derivative of order $\alpha \in (0,1)$,}
\begin{equation}\label{eq:Caputoder}
\partial_t^\alpha u:=\frac{1}{\Gamma(1-\alpha)}\der{}{t}\int_0^tu(s)(t-s)^{-\alpha}ds-u(0)\frac{t^{-\alpha}}{\Gamma(1-\alpha)}.
\end{equation}
This relationship has then been generalized to a general subordinator with Laplace exponent $f(\lambda)$ \mg{in} different ways \cite{chen2017time,hernandez2017generalised,kolokoltsov2009generalized,magdziarz2015asymptotic,meerschaert2008triangular,orsingher2016time,orsingher2018semi,toaldo2015levy}. In particular in \cite{chen2017time} the author proved that, when $f$ is a general Bernstein function having representation \eqref{eq:Bernfun}, then the function $q(x,t)=\E^x[u(X^f(t))]$ satisfies:
\begin{equation}\label{eq:Geneq}
\partial_t\int_0^t(\mg{q(x,s)-q(x,0)})\overline{\nu}(t-s)ds=G_Mq(x,t), \quad q(x,0)=u(x)\in \Dom(G_M)
\end{equation}
where $\overline{\nu}(t)=\nu(t,\infty)$. Further he proved that the occupation measure of $X^f$ is always infinite when the subordinator has infinite expectation. We further observe that when $M$ has the strong Markov property then the process $X^f$ is a semi-Markov process in the sense of \cite[Section 4b]{cinlar1974markov} or \cite{meerschaert2014semi}, i.e. it is not Markovian but it enjoys the Markov property at any random time $\mathcal{T}$ such that \mg{$\mathcal{T}(\omega)\in \{s: \sigma(y,\omega)=s \mbox{ for some }y\}$}.
\subsection{Asymptotic behaviour of the tail}
\bt{In this section we provide} an estimate of the tail of the distribution of the first exit time from an open set of the time-changed process $X^f$. Remark that if $M(y)$ and $\sigma(y)$ are self-similar, $M(y)$ in defined in $\R$ and we consider as particular open set the interval $(0,+\infty)$, explicit results on such distribution are given in \cite{loeffen}.\\
\bt{Let's first introduce the following notation which will be used all throughout the paper. Reserve $S \in \fG$ for an arbitrary open set and define
\begin{equation}\label{eq:firstpasstime}
T:=\inf\{y\ge 0: \ M(y)\not \in S\}
\end{equation}
and
\begin{equation}
\fT:=\inf\{t \ge 0: \ X^f(t)\not \in S\}.
\end{equation}
Further, in order to avoid trivialities, in what follows the results will be always stated for $x$ such that
\begin{align}
P^x(T>0)>0.
\end{align}
In the forthcoming proofs we will make use of the following easy technical lemma.}
\begin{lem}\label{lem:teclem1}
Let $\mathcal{T}$ be a non-negative random variable, $\lambda>0$ and $\mathbf{e}_\lambda$ an exponential random variable of parameter $\lambda$ which is independent from $\mathcal{T}$. Then:
\begin{equation}\label{eq:teclem1}
\bP^x(\mathcal{T}>\mathbf{e}_\lambda)=\E^x[1-e^{-\lambda \mathcal{T}}].
\end{equation}
\end{lem}
\begin{proof}
We only need to observe that:
\begin{align*}
\E^x[1-e^{-\lambda \mathcal{T}}]&=\int_0^{+\infty}(1-e^{-\lambda t})d\bP^x(\mathcal{T}<t)\\&=[(1-e^{-\lambda t})\bP^x(\mathcal{T}<t)]_0^{+\infty}-\int_0^{+\infty}\lambda e^{-\lambda t}\bP^x(\mathcal{T}<t)dt\\&=1-\int_0^{+\infty}\lambda e^{-\lambda t}(1-\bP^x(\mathcal{T}>t))dt\\&=\int_0^{+\infty}\lambda e^{-\lambda t}\bP^x(\mathcal{T}>t)dt=\bP^x(\mathcal{T}>\mathbf{e}_\lambda).
\end{align*}
\end{proof}
By using this lemma, we can show the following result.
\begin{thm}\label{thm:asymbeh}
\bt{Let $x \in S$ be such that} the function $g(s):=\E^x[1-e^{-sT}]$ is regularly varying at zero with index $\beta \in [0,1]$ and $f$ is regularly varying at zero with index $\alpha \in [0,1)$, then $t \mapsto \bP^x(\fT>t)$ varies regularly at infinity with index $\alpha\beta$ and
\begin{equation}\label{eq:asympbeh}
\bP^x(\fT>t)\sim \frac{1}{\Gamma(1-\alpha\beta)}g\left(f\left(\frac{1}{t}\right)\right).
\end{equation}
\end{thm}
\begin{proof}
Let us first note that $X^f$ can be expressed equivalently as
\begin{equation}\label{eq:diffexp}
X^f(t)=M(y), \quad \sigma(y-)\le t < \sigma(y).
\end{equation}
Hence we have that, on any path, $\fT=\sigma(T-)$. But \bt{we know by \cite[Lemma 2.3.2]{applebaum2009levy} that $\sigma$ has no fixed discontinuities, i.e., for any fixed $t >0$ it is true that $\sigma(t-)=\sigma(t)$ a.s. and thus we can write by a conditioning argument
\begin{equation}\label{eq:thmasymbehpass1}
\bP^x(\fT>y)=\mathds{E}^x\bP^x(\sigma(T-)>y|T)=\mathds{E}^x\bP^x(\sigma(T)>y|T) = \bP^x(\sigma(T)>y).
\end{equation}}
Furthermore by definition of $L$ we have that $L(\sigma(t))=t$ a.s. and thus we can rewrite \eqref{eq:thmasymbehpass1} as
\begin{equation}\label{eq:thmasymbehpass2}
\bP^x(\fT>y)=\bP^x(T>L(y)).
\end{equation}
Now let
\begin{equation}\label{eq:thmasymbehU}
U(t):=\int_0^t\bP^x(\fT>y)dy
\end{equation}
and
\begin{equation}\label{eq:thmasymbehtilU}
\widetilde{U}(\lambda):=\int_0^{+\infty}e^{-\lambda t}U(dt).
\end{equation}
Let $\cL(t)$ be a slowly varying function at infinity. The Karamata's Tauberian Theorem \cite[Thm XIII.5.2]{feller1968introduction} states that the relations
\begin{equation}\label{eq:thmasymbehpass3}
U(t)\sim \frac{t^\rho}{\Gamma(1+\rho)}\cL(t) \mbox{ as }t \to +\infty
\end{equation}
and
\begin{equation}\label{eq:thmasymbehpass4}
\widetilde{U}(\lambda)\sim \lambda^{-\rho}\cL(1/\lambda) \mbox{ as }\lambda \to 0
\end{equation}
imply each other. Now we need to determine the relation \eqref{eq:thmasymbehpass4} for \eqref{eq:thmasymbehtilU}. By using \eqref{eq:thmasymbehpass2} we find that
\begin{align}\label{eq:thmasymbehpass5}
\begin{split}
\widetilde{U}(\lambda)&=\int_0^\infty e^{-\lambda t}U(dt)\\&=\int_0^\infty e^{-\lambda t}\int_0^{\infty}\bP^x(T>w)\bP^x(L(t)\in dw)dt.
\end{split}
\end{align}
By \cite[eq (3.13)]{meerschaert2008triangular} we further have that $L(t)$ has a Lebesgue density $s \mapsto l(s,t)$ such that
\begin{equation}\label{eq:thmasymbehpass6}
\int_0^{\infty}e^{-\lambda t}l(s,t)dt=\frac{f(\lambda)}{\lambda}e^{-sf(\lambda)}.
\end{equation}
Hence we can write
\begin{equation}\label{eq:thmasymbehpass7}
\widetilde{U}(\lambda)=\frac{1}{\lambda}\int_0^{\infty}f(\lambda)e^{-yf(\lambda)}\bP^x(T>y)dy.
\end{equation}
Consider now an exponential random variable $\mathbf{e}_{f(\lambda)}$ of parameter $f(\lambda)$ and independent of $T$. Thus we have
\begin{equation}\label{eq:thmasymbehpass8}
\widetilde{U}(\lambda)=\frac{1}{\lambda}\bP^x(T>\mathbf{e}_{f(\lambda)}).
\end{equation}
Then, by using Lemma \ref{lem:teclem1} we have
\begin{equation}\label{eq:thmasymbehpass9}
\widetilde{U}(\lambda)=\frac{1}{\lambda}\E^x[1-e^{-f(\lambda)T}]=\frac{1}{\lambda}g(f(\lambda)).
\end{equation}
Now let us observe that, by hypotheses, $g$ is regularly varying at $0$ with index $\beta \in [0,1]$ and $f$ is regularly varying at $0$ with index $\alpha \in [0,1)$, so $g \circ f$ is regularly varying at $0$ with index $\alpha \beta \in [0,1)$ \bt{by an application of \cite[Proposition 1.5.7]{bingham1989regular}}. Thus there exists a function $\cL$ which is slowly varying at $0$ such that
\begin{equation}\label{eq:thmasymbehpass10}
g(f(\lambda))=\lambda^{\alpha\beta}\cL(\lambda)
\end{equation}
and thus Eq. \eqref{eq:thmasymbehpass9} becomes
\begin{equation}\label{eq:thmasymbehpass11}
\widetilde{U}(\lambda)=\lambda^{\alpha\beta-1}\cL(\lambda).
\end{equation}
By Karamata's Tauberian theorem we mentioned before we have as $t \to \infty$
\begin{equation}\label{eq:thmasymbehpass12}
U(t)\sim \frac{t^{1-\alpha\beta}}{\Gamma(2-\alpha\beta)}\cL(1/t).
\end{equation}
Applying then the Monotone Density Theorem \cite[Thm 1.7.2]{bingham1989regular} we have as $t \to \infty$
\begin{equation}\label{eq:thmasymbehpass13}
\bP^x(\fT>t)\sim \frac{1-\alpha\beta}{\Gamma(2-\alpha\beta)}t^{-\alpha\beta}\cL(1/t)=\frac{1}{\Gamma(1-\alpha\beta)}g(f(1/t))
\end{equation}
where we used Eq. \eqref{eq:thmasymbehpass10} and the fact that $z\Gamma(z)=\Gamma(z+1)$.
\end{proof}
Since checking that $g(s)=\E^x[1-e^{-sT}]$ is regularly varying \bt{may be a difficult task}, we propose the following corollary.
\begin{cor}\label{cor:corasymbeh}
If, for some $x \in S$, $\E^x[T]=C<+\infty$ and $f$ is regularly varying at zero with index $\alpha \in [0,1)$, then $t \mapsto \bP^x(\fT>t)$ varies regularly at infinity and
\begin{equation}\label{eq:asympwithfinmean}
\bP^x(\fT>t)\sim \frac{C}{\Gamma(1-\alpha)}f\left(\frac{1}{t}\right).
\end{equation}
\end{cor}
\begin{proof}
First let us observe that
\begin{equation*}
1-e^{-sT}\le sT
\end{equation*}
for any $s \in \R$. Thus we have
\begin{equation*}
\frac{1-e^{-sT}}{s}\le T
\end{equation*}
for any $s \in \R$. Moreover \bt{since we assumed} that $T$ is an integrable random variable \bt{we have by dominated convergence that}
\begin{equation*}
\lim_{s \to 0}\frac{g(s)}{s}=\lim_{s \to 0}\E^{x}\left[\frac{1-e^{-sT}}{s}\right]=\E^{x}\left[\lim_{s \to 0}\frac{1-e^{-sT}}{s}\right]=\E^x[T]=C
\end{equation*}
\bt{from which we get}
\begin{equation*}
g(s)\sim Cs.
\end{equation*}
Thus $g(s)$ is regularly varying at $0$ with index $1$ and then we can use Theorem \ref{thm:asymbeh} \bt{to say that}
\begin{equation*}
\bP(\fT>t)\sim \frac{1}{\Gamma(1-\alpha)}g\left(f\left(\frac{1}{t}\right)\right).
\end{equation*}
Finally let us observe that $g\left(f\left(\frac{1}{t}\right)\right)\sim Cf\left(\frac{1}{t}\right)$ to obtain Eq. \eqref{eq:asympwithfinmean}.
\end{proof}
Let us \bt{see some instructive} examples.
\begin{exmp}\label{ex:BMd}
Consider $W_{\delta}(t)=W(t)+\delta t$ a $1$-dimensional Wiener process with positive drift \bt{$\delta>0$} (where $W(t)$ is a standard Wiener process) and the open set $S=(-\infty,c)$ for $c>0$. Consider $T:=\inf\{t \ge 0: \ W_\delta(t)\not \in S\}$ and observe that $T$ is absolutely continuous with probability density function $p_T(t)\bt{dt}=\bP^0(T \in dt)$ given by \bt{(e.g. \cite[eq. 2.0.2, pag 295]{BMHB})}
\begin{equation}
p_T(t)=\frac{c}{\sqrt{2\pi}}\frac{e^{-\frac{(c-\delta t)^2}{2t}}}{t^{\frac{3}{2}}}1_{(0,+\infty)}(t).
\end{equation}
It is \bt{well-known that $\E^0[T]=\frac{c}{\delta}<+\infty$}. Consider then $W_\delta^f(t):=W_\delta(L(t))$ and $\fT:=\inf\{t \ge 0: \ W_\delta^f(t) \not \in S\}$. Thus, by Corollary \ref{cor:corasymbeh}, if $f$ is regularly at zero varying with index $\alpha \in (0,1)$ we know that
\begin{equation}
\bP^0(\fT > t)\sim \frac{c}{\delta\Gamma(1-\alpha)}f\left(\frac{1}{t}\right).
\end{equation}
\end{exmp}
\begin{exmp}\label{ex:BM}
Consider $W(t)$ a $1$-dimensional standard Wiener process and the open set $S=(-\infty,c)$ for $c>0$. Consider $T:=\inf\{t \ge 0: W(t)\not \in S\}$ and observe that $T$ is absolutely continuous with probability density function $p_T(t)\bt{dt}=\bP^0(T \in dt)$ given by \bt{(e.g. \cite[eq. 2.0.2 pag. 198]{BMHB})}
\begin{equation}
p_T(t)=\frac{c}{\sqrt{2\pi}}\frac{e^{-\frac{c^2}{2t}}}{t^{\frac{3}{2}}}.
\end{equation}
In this case $\E^0[T]=+\infty$ so we cannot use Corollary \ref{cor:corasymbeh}. Thus we want to study the function $g(s)=\E^0[1-e^{-sT}]$. To do this, let us introduce a Lévy subordinator $\tau(t)$, that is to say a $1/2$-stable subordinator, with probability density function
\begin{equation}
p_{\tau(t)}(\lambda)=\frac{t}{2\sqrt{\pi}}\frac{e^{-\frac{t^2}{4\lambda}}}{\lambda^{\frac{3}{2}}}.
\end{equation}
For this process we know that
\begin{equation}
e^{-t\sqrt{s}}=\E[e^{-s\tau(t)}]=\int_0^\infty e^{-s\lambda}\frac{t}{2\sqrt{\pi}}\frac{e^{-\frac{t^2}{4\lambda}}}{\lambda^{\frac{3}{2}}}d\lambda.
\end{equation}
Thus, let us observe that, by using the change of variable $x=2y$
\begin{align}
\begin{split}
\E[e^{-sT}]&=\int_0^\infty e^{-sx}\frac{c}{\sqrt{2\pi}}\frac{e^{-\frac{c^2}{2x}}}{x^{\frac{3}{2}}}dx\\&=\int_0^\infty e^{-2sy}\frac{c}{2\sqrt{\pi}}\frac{e^{-\frac{c^2}{4y}}}{y^{\frac{3}{2}}}dy=\E[e^{-2s\tau(c)}]=e^{-c\sqrt{2s}}.
\end{split}
\end{align}
Then we have
\begin{equation}
g(s)=\E[1-e^{-sT}]=1-e^{-c\sqrt{2s}}
\end{equation}
which is a regularly varying function at $0^+$ with index $1/2$. Consider now $W^f(t):=W(L(t))$ and $\fT:=\inf\{t \ge 0: \ W^f(t)\not \in S\}$. Thus, by Theorem \ref{thm:asymbeh}, if $f$ is regularly varying at zero with index $\alpha\in (0,1)$ we know that
\begin{equation}
\bP^0(\fT>t)\sim \frac{1}{\Gamma\left(1-\frac{\alpha}{2}\right)}\left[1-e^{-c\sqrt{2f(1/t)}}\right].
\end{equation}
\end{exmp}
The following proposition shows a particular case of Theorem \ref{thm:asymbeh} in which the distribution of $\fT$ can be computed explicitly.
\begin{prop}\label{eq:exptoML}
Let $f(\lambda)=\lambda^\alpha$. If $\bP^x(T>y)=e^{-hy}$ for some $h \ge 0$ then we have
\begin{equation}\label{eq:fptstablesub}
\bP(\fT>t)=E_\alpha(-ht^\alpha):=\sum_{k=0}^{\infty}\frac{(-ht^\alpha)^k}{\Gamma(\alpha k+1)}.
\end{equation}
Furthermore
\begin{equation}\label{eq:fptasympstablesub}
\bP^x(\fT>t)\sim \frac{1}{h}\frac{t^{-\alpha}}{\Gamma(1-\alpha)}
\end{equation}
as $t \to +\infty$.
\end{prop}
\begin{proof}
By using Eq. \eqref{eq:thmasymbehpass2} we have
\begin{equation}
\bP^x(\fT>t)=\int_0^{\infty}\bP^x(T >y)\bP^x(L(t)\in dy)=\int_0^\infty e^{-hy}\bP^x(L(t)\in dy).
\end{equation}
As proved in \cite{bingham1971limit} we know that the Laplace transform of the inverse of an $\alpha$-stable subordinator is
\begin{equation}
\int_0^{\infty}e^{-hy}\bP^x(L(t)\in dy)=E_\alpha(-ht^\alpha)
\end{equation}
and this proves the first statement. The second statement is a consequence of Corollary \ref{cor:corasymbeh} since $\E^x[T]=\frac{1}{h}$, but the fact that
\begin{equation}
E_\alpha(-ht^\alpha)\sim \frac{1}{h}\frac{t^{-\alpha}}{\Gamma(1-\alpha)}
\end{equation}
as $t\to +\infty$ is a well-known fact (e.g. \cite[eq. (24)]{scalas2006five}).
\end{proof}
\begin{rmk}
\bt{Assuming that $f(\lambda)$ is regularly varying at zero with $\alpha \in [0,1)$ implies that the corresponding subordinator has infinite expectation. Further,} since we have by eq. \eqref{eq:thmasymbehpass12} that $U(t)$ defined in Eq. \eqref{eq:thmasymbehU} varies regularly at infinity with index $1-\beta\alpha>0$, it follows that $\E^x[\fT]=+\infty$, hence our result agree with \cite[Thm 3.1]{chen2017time}.
\end{rmk}

\textcolor{black}{
By using Theorem \ref{thm:asymbeh} we can show the following two results concerning family of open sets.
\begin{prop}
Let $\{S_t,t\ge 0\} \subseteq \fG$ be a family of open sets such that $\bigcap_{t \ge 0}S_t \not = \emptyset$ and suppose there exists an open set $S \supseteq \bigcup_{t \ge 0}S_t$ such that $T$ is almost surely finite. Then, if for some $x \in \bigcap_{t \ge 0}S_t$ the function $g(s)=\E^x[1-e^{-sT}]$ is regularly varying at $0^+$ with index $\beta \in [0,1]$ and $f(\lambda)$ is regularly varying at $0^+$ with index $\alpha \in [0,1)$,
\begin{equation}\label{eq:limsupsurv}
\limsup_{t \to +\infty}\frac{\bP^x(\widehat{\fT}>t)\Gamma(1-\alpha \beta)}{g(f(1/t))} \le 1
\end{equation}
where
\begin{equation}\label{eq:fTapvar}
\widehat{\fT}:=\inf\{t>0: \ X^f(t) \not \in S_t \}
\end{equation}
\end{prop}
\begin{proof}
Let us observe that $S_{\sigma(y)}\subseteq S$ and $S_{\sigma(y-)}\subseteq S$. Then we have $\fT \ge \widehat{\fT}$ and
\begin{equation*}
\bP^x(\widehat{\fT}>t)\le \bP^x(\fT>t).
\end{equation*}
Hence we have
\begin{equation*}
\frac{\bP^x(\widehat{\fT}>t)\Gamma(1-\alpha\beta)}{g(f(1/t))}\le \frac{\bP^x(\fT>t)\Gamma(1-\alpha\beta)}{g(f(1/t))}
\end{equation*}
and then, taking the $\limsup_{t \to +\infty}$ and using Theorem \ref{thm:asymbeh}, we obtain Equation \eqref{eq:limsupsurv}. 
\end{proof}
\begin{prop}
Let $\{S_t,t\ge 0\} \subseteq \fG$ be a family of open sets and suppose there exists an open set \mg{$\emptyset \not = S \subseteq \bigcap_{t \ge 0}S_t$} such that $T$ is almost surely finite. Then, if for some $x \in S$ the function $g(s)=\E^x[1-e^{-sT}]$ is regularly varying at $0^+$ with index $\beta \in [0,1]$ and $f(\lambda)$ is regularly varying at $0^+$ with index $\alpha \in [0,1)$,
\begin{equation}\label{eq:liminfsurv}
\liminf_{t \to +\infty}\frac{\bP(\widehat{\fT}>t)\Gamma(1-\alpha \beta)}{g(f(1/t))} \ge 1
\end{equation}
where $\widehat{\fT}$ has been defined before in Eq. \eqref{eq:fTapvar}.
\end{prop}
\begin{proof}
Let us observe that $S_{\sigma(y)}\supseteq S$ and $S_{\sigma(y-)}\supseteq S$. Then we have $\fT \le \widehat{\fT}$ and
\begin{equation*}
\bP(\widehat{\fT}>t)\ge \bP(\fT>t).
\end{equation*}
Hence we have
\begin{equation*}
\frac{\bP(\widehat{\fT}>t)\Gamma(1-\alpha\beta)}{g(f(1/t))}\ge \frac{\bP(\fT>t)\Gamma(1-\alpha\beta)}{g(f(1/t))}
\end{equation*}
and then, taking the $\liminf_{t \to +\infty}$ and using Theorem \ref{thm:asymbeh}, we obtain Equation \eqref{eq:limsupsurv}. 
\end{proof}
}
\subsection{\bt{Smoothness}}
In the previous section we have used the Monotone Density Theorem to deduce the asymptotic behaviour at infinity of the function $t \mapsto \bP^x(\fT>t)$. Moreover we could use such theorem if $\fT$ is absolutely continuous to deduce the asymptotic behaviour of the probability density function of $\fT$. For this reason, it could be interesting to investigate what are some assumptions under which $\fT$ is absolutely continuous.
\begin{thm}\label{thm:abscont}
If the function $s \mapsto \overline{\nu}(s)$ is absolutely continuous, then $\fT$ is an absolutely continuous random variable.
\end{thm}
\begin{proof}
\bt{Note that absolute continuity of $ s \mapsto \bar{\nu}(s)$ together with $\nu (0, \infty)=\infty$ imply, by \cite[Theorem 27.10]{sato1999levy}, that $\sigma(t)$ has a Lebesgue density $\mu(x,t)$.}
Indeed let us recall, from \cite[Def. 27.9]{sato1999levy}, that a measure $\nu$ on $\R^d\setminus\{0\}$ is radially absolutely continuous if there are a finite measure $\lambda$ on the unit sphere $S$ of $\R^d$ and a non-negative measurable function $(\xi,r)\in S \times (0,\infty) \mapsto g(\xi,r) \in \R$ such that for any Borel set $B$ of $\R^d\setminus\{0\}$
\begin{equation}
\nu(B)=\int_{S}\lambda(d\xi)\int_0^\infty g(\xi,r)1_B(r\xi)dr.
\end{equation}
If $d=1$, then let us observe that $S=\{-1,1\}$. Since $s \mapsto \overline{\nu}(s)$ is absolutely continuous, there exists a function $g(1,s)$ such that $\overline{\nu}(s)=\int_s^{+\infty}g(1,r)dr$. Let us also pose $g(-1,r)=0$ for any $r \in (0,\infty)$. Moreover let us pose $\lambda=\delta_1$ where $\delta_1$ is the Dirac delta centred in $1$. Then it is easy to see that
\begin{equation}
\nu(B)=\int_{S}\lambda(d\xi)\int_0^\infty g(\xi,r)1_B(r\xi)dr=\int_{B \cap (0,+\infty)}g(1,r)dr.
\end{equation}
Moreover, since $\nu(0,\infty)=+\infty$, then $\int_0^{+\infty}g(1,r)dr=+\infty$ and $\nu$ satisfies also the divergence condition. 

\bt{Using} \eqref{eq:thmasymbehpass1} we can write
\begin{equation}
\bP^x(\fT \in ds)=\int_{0}^{\infty}\bP^x(\sigma(w)\in ds)\bP^x(T \in dw)=\int_0^{\infty}\mu(s,w)\bP^x(T \in dw)
\end{equation}
\bt{anf thus} $\fT$ is absolutely continuous with probability density function 
\begin{equation}\label{eq:pdffT}
p^x_{\fT}(s)=\int_0^{\infty}\mu(s,w)\bP^x(T \in dw).
\end{equation}
\end{proof}
\bt{We can further} investigate conditions under which the probability density function $p_{\fT}$ is infinitely differentiable.
\begin{prop}\label{prop:infdiff}
If $s \mapsto \overline{\nu}(s)$ is absolutely continuous and there exist $\gamma \in (0,2)$, $C>0$ and $r_0>0$ such that
\begin{equation}\label{eq:hyp1}
\int_0^rs^2\nu(ds)>Cr^\gamma, \quad \mbox{ for all }0<r<r_0,
\end{equation}
then $\fT$ is an absolutely continuous random variable and its probability density function $p_{\fT}$ is infinitely differentiable.
\end{prop}
\begin{proof}
The fact that $\fT$ is absolutely continuous is consequence of Theorem \ref{thm:abscont}. Moreover $p_{\fT}$ is given by Eq. \eqref{eq:pdffT}. Under hypothesis \eqref{eq:hyp1}, by using the results in \cite{orey1968continuity}, we know that for some $c>0$ and $\xi$ sufficiently large
\begin{equation}
\left|\E\left[e^{i\xi\sigma(1)}\right]\right|\le e^{-\frac{c}{4}|\xi|^{2-\gamma}}.
\label{oreybound}
\end{equation}
\bt{and thus one can differentiate under integration in
\begin{align}
\mu(x, t) \, =  \, \frac{1}{2\pi} \int_{\mathbb{R}}  e^{-i\xi x} e^{-t\varphi(\xi)}d\xi
\label{263}
\end{align}
where we denote by $\varphi$ the L\'evy symbol of $\sigma$. Recall now that, \mg{from Eq. \eqref{eq:pdffT}}
\begin{align}
\bP^x \l \mathfrak{T}  \in ds\r /ds \, = \, \int_0^\infty \mu(s, w) \, \bP^x \l T \in dw \r.
\end{align}
Use \eqref{263} to say that
\begin{align}
\bP^x \l \mathfrak{T}  \in ds\r /ds \, = \, \frac{1}{2\pi} \int_{\mathbb{R}} e^{-i\xi s} \int_0^\infty e^{-w\varphi(\xi)}\bP^x(T \in dw) \, d\xi
\label{invfour}
\end{align}
and note that
\begin{align}
\left|\frac{e^{-i\xi s}-e^{-i\xi s^\prime}}{s-s^\prime} \int_0^\infty e^{-w \varphi(\xi)}\bP^x (T \in dw)\right| \leq  \int_0^\infty \left| \xi e^{-w\varphi(\xi)}\right|P^x \l T \in dw \r.
\end{align}
Hence by using \eqref{oreybound} we can apply dominated convergence to differentiate repeatedly under integration and thus} $p_{\fT}(s)$ is infinitely differentiable.
\end{proof}
If we know that $p_{\fT}$ admits derivatives of all order, then we could be interested in when such derivatives admit Laplace transform. A particular case could be the one in which we can prove that all the derivatives of $p_{\fT}$ are bounded. In particular we can show the following Proposition
\begin{prop}\label{prop:Laptrans}
\bt{Under the assumptions of Proposition \ref{prop:infdiff} the density $p_{\fT}$ and all its derivatives are bounded}.
\end{prop}
\begin{proof}
\bt{Use \eqref{invfour} to say that for $s \in \mathbb{R}$
\begin{align}
\left|\frac{\partial^n}{\partial s^n} p_{\fT} (s) \right| \, \leq  \, \int_\mathbb{R} \int_0^\infty |\xi|^n \left|e^{-w\varphi(\xi)}\right| \bP^x(T \in dw) d\xi
\label{unifb}
\end{align}
and use \eqref{oreybound} to say that the right-hand side of \eqref{unifb} is finite. Since this bound does not depend on $s$ the result is proved.}
%
\end{proof}
\subsection{Rapid behaviour at zero}
In order to determine some properties related to the asymptotic behaviour at $0$ of the distribution function of $\fT$, one has also to work with functions whose decay at $0$ is more rapid then any power function.\\
Let us say that a function $f:[0,+\infty[\to [0,+\infty[$ is \textit{rapidly decreasing at $0^+$} if:
\begin{equation}
\forall \alpha >0, \ \lim_{t \to 0^+}\frac{f(t)}{t^\alpha}=0.
\end{equation}
It follows from the definition that in such case $\lim_{t \to 0^+}f(t)=0$.
About regularity in $0$ of such functions, we have the following Lemma.
\begin{lem}\label{lem:regrapdecfun}
Suppose $f \in C^{\infty}(0,\delta)$ for some $\delta>0$. \bt{Then the following are equivalent:
\begin{enumerate}
\item $f$ is rapidly decreasing at $0^+$
\item  $f \in C^{\infty}([0,\delta))$ and $f^{(n)}(0)=0$ for all $n \in \N$.
\end{enumerate}}
Moreover if $f$ is rapidly decreasing at $0^+$ then all its derivatives are rapidly decreasing at $0^+$.
\end{lem}
\begin{proof}
First let us suppose that $f$ is rapidly decreasing at $0^+$ and let us show that $f \in C^{\infty}([0,\delta))$ and $f^{(n)}(0)=0$ for all $n \in \N$. We will show it by induction. Let us first notice that $f(0)=\lim_{t \to 0^+}f(t)=0$. Then let us notice that by definition
\begin{equation}
f'(0):=\lim_{t \to 0^+}\frac{f(t)}{t}=0.
\end{equation}
Now suppose $f \in C^n([0,\delta))$ and for any $m \le n$ we have $f^{(m)}(0)=0$. Thus we can use L'Hopital rule on $\lim_{t \to 0}\frac{f^{(m)}(t)}{t^\alpha}$ for any $m \le n$ and $\alpha>0$. In particular we have
\begin{equation*}
0=\lim_{t \to 0^+}\frac{f(t)}{t^{n+1}}=\lim_{t \to 0^+}\frac{f'(t)}{(n+1)t^{n}}=\dots=\lim_{t \to 0^+}\frac{f^{(n+1)}(t)}{(n+1)!}
\end{equation*}
and then we have $f^{(n+1)}(0)=0$.\\
Now suppose $f \in C^{\infty}([0,\delta))$ with $f^{(n)}(0)=0$ for any $n \in \N$ and let us show that $f$ is rapidly decreasing at $0^+$. First fix $n \in \N$ and observe that, by l'Hopital rule:
\begin{equation*}
\lim_{t \to 0^+}\frac{f(t)}{t^n}=\dots= \lim_{t \to 0^+}\frac{f^{(n)}(t)}{n!}=\frac{f^{(n)}(0)}{n!}=0.
\end{equation*}
Now consider a generic $\alpha>0$ and fix $n=\lfloor \alpha \rfloor+1$. Since $n \in \N$, we know that $\lim_{t \to 0^+}\frac{f(t)}{t^n}=0$ and $n-\alpha>0$. Thus we have
\begin{equation*}
\lim_{t \to 0^+}\frac{f(t)}{t^\alpha}=\lim_{t \to 0^+}\frac{f(t)}{t^n}t^{n-\alpha}=0.
\end{equation*}
Finally, let us observe that if $f \in C^{\infty}(0,\delta)$ is rapidly decreasing at $0^+$, then we have that $f \in C^{\infty}([0,\delta))$ and $f^{(n)}(0)=0$ for all $n \in \N$. Fix $m \in \N$ and observe that $f^{(m)}\in C^{\infty}([0,\delta))$ and for all $n \in \N$ we also have $f^{(m+n)}(0)=0$, so $f^{(m)}$ is rapidly decreasing at $0^+$.
\end{proof}
To study the asymptotic behaviour of the distribution function of $\fT$ near infinity we used the Tauberian theorem for regularly varying functions. Thus we will need a sort of Tauberian theorem also for rapidly decaying functions.
\begin{lem}\label{lem:Taubrapid}
Let $f \in C^{\infty}(0,\infty)$ and suppose $f$ and all its derivatives admit Laplace transform. Denote with $\widetilde{f}$ the Laplace transform of $f$. Then $f$ is rapidly decreasing at $0^+$ if and only if \bt{$\lim_{\lambda \to \infty}\lambda^\alpha \widetilde{f}(\lambda)=0$ for any $\alpha >0$.}
\end{lem}
\begin{proof}
Let us first show that if $f$ is rapidly decreasing at $0^+$ then we have $\lambda^\alpha \widetilde{f}(\lambda)\to 0$, as $\lambda \to +\infty$, for all $\alpha >0$. Now note that, by the Initial-Value Theorem (e.g. \cite[Section 17.8]{cannon2003dynamics}), we have
\begin{equation}
\lim_{\lambda \to \infty}\lambda\widetilde{f}(\lambda)=\lim_{t \to 0^+}f(t)=0.
\end{equation}
Now fix $n \in \N$ with $n>1$ and denote by $\cL$ the Laplace transform operator. Since $f$ is rapidly decreasing at $0^+$, by Lemma \ref{lem:regrapdecfun} we know that $f^{(n-1)}(0)=0$. Moreover by hypothesis we know that $f^{(n-1)}$ admits Laplace transform and then, since for any $k\le n-1$ $f^{(k)}(0)=0$,
\begin{equation}
\cL[f^{(n-1)}](\lambda)=\lambda^{n-1}\widetilde{f}(\lambda).
\end{equation}
Thus, by the Initial-Value Theorem, we have
\begin{equation}
\lim_{\lambda \to \infty}\lambda^n\widetilde{f}(\lambda)=\lim_{\lambda \to \infty}\lambda \lambda^{n-1}\widetilde{f}(\lambda)=f^{(n-1)}(0)=0.
\end{equation}
Finally let us consider a generic $\alpha>0$. Pose $n=\lfloor \alpha \rfloor+1$ so that $n \in \N$ and $\alpha-n<0$. Thus we have
\begin{equation}
\lim_{\lambda \to \infty}\lambda^\alpha \widetilde{f}(\lambda)=\lim_{\lambda \to \infty}\lambda^{\alpha-n}\lambda^n\widetilde{f}(\lambda)=0.
\end{equation}
Now let us show that if for any $\alpha>0$ we have $\lim_{\lambda \to \infty}\lambda^\alpha \mg{\widetilde{f}}(\lambda)=0$ then $f$ is rapidly decreasing at $0^+$. To do this, let us proceed by induction. First observe that
\begin{equation}
f(0)=\lim_{\lambda \to \infty}\lambda\widetilde{f}(\lambda)=0.
\end{equation}
Now observe that, since $f(0)=0$, we have
\begin{equation}
\cL[f'](\lambda)=\lambda\widetilde{f}(\lambda)
\end{equation}
thus, by the Initial-Value Theorem
\begin{equation}
f'(0)=\lim_{\lambda \to \infty}\lambda\cL[f'](\lambda)=\lim_{\lambda \to \infty}\lambda^2\widetilde{f}(\lambda)=0.
\end{equation}
Now fix $n>1$ and suppose that $f^{(k)}(0)=0$ for any $k<n$. Then we have that
\begin{equation}
\cL[f^{(n)}](\lambda)=\lambda^n\widetilde{f}(\lambda).
\end{equation}
Thus, by the Initial-Value Theorem we have
\begin{equation}
f^{(n)}(0)=\lim_{\lambda \to \infty}\lambda\cL[f^{(n)}](\lambda)=\mg{\lim_{\lambda \to +\infty}\lambda^{n+1}\widetilde{f}(\lambda)=0}.
\end{equation}
Since we have shown that $f^{(n)}(0)=0$ for any $n \in \N$ we have, by Lemma \ref{lem:regrapdecfun}, that $f$ is rapidly decreasing at $0^+$.
\end{proof}
\subsection{Asymptotic behaviour of the distributions at zero}
Here we want to provide an estimate near $0$ of the distribution of the first exit time from an open set of the time-changed process $X^f$. This time we need the distribution function $\bP^x(T\le t)$ to be regular varying at zero.\bt{ We will always use the notation
\begin{align}
&F^x(t) \, : = \, \bP^x \l T \leq t \r \\
& \mathfrak{F}^x(t) \, : = \, \bP^x \l \mathfrak{T} \leq t \r.
\end{align}}
\begin{thm}\label{thm:asymbehzero}
If, for some $x \in S$, the function $F^x(t)$ varies regularly at zero with index $\rho>0$ and $f(\lambda)$ varies regularly at infinity with index $\alpha>0$, then $\fF(t)$ varies regularly at zero with index $\alpha \rho$ and as $t \to 0^+$
\begin{equation}
\fF^x(t)\sim \frac{\Gamma(1+\rho)}{\Gamma(1+\alpha\rho)}F^x\left(\frac{1}{f\left(\frac{1}{t}\right)}\right).
\end{equation}
\end{thm}
\begin{proof}
Let us define
\begin{equation}\label{eq:Uforzero}
\widetilde{F}(\lambda):=\int_0^\infty e^{-\lambda t}dF^x(t)
\end{equation}
and
\begin{equation}\label{eq:fUforzero}
\widetilde{\fF}(\lambda):=\int_0^\infty e^{-\lambda t}d\fF^x(t).
\end{equation}
Since $F^x(t)$ varies regularly at zero with index $\rho>0$, by Tauberian theorems \cite[Theorem XIII.5.2 and XIII.5.3]{feller1968introduction}, we have that $\widetilde{F}(\lambda)$ varies regularly at infinity with index $\rho$ and as $\lambda \to \infty$
\begin{equation}\label{eq:asymbehzeropass1}
\widetilde{F}(\lambda)\sim F^x\left(\frac{1}{\lambda}\right)\Gamma(1+\rho).
\end{equation}
Recalling Eqs. \eqref{eq:thmasymbehpass1} and \eqref{eq:fUforzero} we obtain
\begin{align}\label{eq:LaptransfFeqLaptransFoff}
\widetilde{\fF}(\lambda) \notag \,  = \, & \int_0^\infty \int_0^\infty e^{-\lambda t}\bP^x(\sigma(s)\in dt)\bP^x(T \in ds) \notag \\= \, & \int_0^\infty e^{-sf(\lambda)}\bP^x(T \in ds) \notag \\ = \, & \widetilde{F}(f(\lambda)).
\end{align}
Since $\widetilde{F}(\lambda)$ varies regularly at infinity with index $\rho$ and $f(\lambda)$ varies regularly at infinity with index $\alpha$, then $\widetilde{\fF}(\lambda)=\widetilde{F}(f(\lambda))$ varies regularly at infinity with index $\alpha\rho$ \bt{by \cite[Proposition 1.5.7]{bingham1989regular}}. Moreover, by Eq. \eqref{eq:asymbehzeropass1}, we obtain as $\lambda \to \infty$
\begin{equation}
\widetilde{\fF}(\lambda)\sim \Gamma(1+\rho)F^x\left(\frac{1}{f(\lambda)}\right).
\end{equation}
\bt{Hence, by using again Tauberian theorems} we know that $\fF^x(t)$ varies regularly at zero with index $\alpha \rho$ and as $t \to 0$
\begin{equation}
\fF^x(t)\sim \frac{\Gamma(1+\rho)}{\Gamma(1+\alpha\rho)}F^x\left(\frac{1}{f\left(\frac{1}{t}\right)}\right).
\end{equation}
\end{proof}
\textcolor{black}{\begin{prop}
Let $\{S_t,t\ge 0\} \subseteq \fG$ be a family of open sets such that $\bigcap_{t \ge 0}S_t \not = \emptyset$ and suppose there exists an open set $S \supseteq \bigcup_{t \ge 0}S_t$ such that $T$ is almost surely finite. Then, if for some $x \in \bigcap_{t \ge 0}S_t$ the function $F^x(t)$ is regularly varying at $0^+$ with index $\rho>0$ and $f(\lambda)$ is regularly varying at infinity with index $\alpha>0$,
\begin{equation}\label{eq:liminfdist}
\liminf_{t \to +\infty}\frac{\bP^x(\widehat{\fT}<t)\Gamma(1+\alpha\rho)}{\Gamma(1+\rho)F(1/f(1/t))}\ge 1
\end{equation}
where $\widehat{\fT}$ is defined in \eqref{eq:fTapvar}.
\end{prop}
\begin{proof}
Let us observe that $S_{\sigma(y)}\subseteq S$ and $S_{\sigma(y-)}\subseteq S$. Then we have $\fT \ge \widehat{\fT}$ and
\begin{equation*}
\bP^x(\widehat{\fT}<t)\ge \bP^x(\fT<t).
\end{equation*}
Hence we have
\begin{equation*}
\frac{\bP^x(\widehat{\fT}<t)\Gamma(1+\alpha\rho)}{\Gamma(1+\rho)F(1/f(1/t))}\ge \frac{\bP^x(\fT<t)\Gamma(1+\alpha\rho)}{\Gamma(1+\rho)F(1/f(1/t))}
\end{equation*}
and then, taking the $\liminf_{t \to +\infty}$ and using Theorem \ref{thm:asymbehzero}, we obtain Equation \eqref{eq:liminfdist}. 
\end{proof}
\begin{prop}
Let $\{S_t,t\ge 0\} \subseteq \fG$ be a family of open sets  and suppose there exists an open set $\emptyset\not =S \subseteq \bigcap_{t \ge 0}S_t$ such that $T$ is almost surely finite. Then, if for some $x \in S$ the function $F^x(t)$ is regularly varying at $0^+$ with index $\rho>0$ and $f(\lambda)$ is regularly varying at infinity with index $\alpha>0$,
\begin{equation}\label{eq:limsupdist}
\limsup_{t \to +\infty}\frac{\bP^x(\widehat{\fT}<t)\Gamma(1+\alpha\rho)}{\Gamma(1+\rho)F(1/f(1/t))}\le 1
\end{equation}
where $\widehat{\fT}$ is defined in \eqref{eq:fTapvar}.
\end{prop}
\begin{proof}
Let us observe that $S_{\sigma(y)}\subseteq S$ and $S_{\sigma(y-)}\subseteq S$. Then we have $\fT \le \widehat{\fT}$ and
\begin{equation*}
\bP^x(\widehat{\fT}<t)\le \bP^x(\fT<t).
\end{equation*}
Hence we have
\begin{equation*}
\frac{\bP^x(\widehat{\fT}<t)\Gamma(1+\alpha\rho)}{\Gamma(1+\rho)F(1/f(1/t))}\le \frac{\bP^x(\fT<t)\Gamma(1+\alpha\rho)}{\Gamma(1+\rho)F(1/f(1/t))}
\end{equation*}
and then, taking the $\limsup_{t \to +\infty}$ and using Theorem \ref{thm:asymbehzero}, we obtain Equation \eqref{eq:limsupdist}. 
\end{proof}
}
The previous \bt{result cover the situation in which $F$ is regularly varying at $0$. It will be usefull in the sequel to deal with a rapid decay of $F$ at $0$ and thus in the forthcoming results we take into account this possibility}.
\begin{thm}\label{thm:rapiddecayzero}
Suppose that $T$ and $\fT$ are absolutely continuous with probability density function $p^x_T(t)=\bP^x(T \in dt)\bt{/ dt}$ and $p^x_{\fT}\bt{(t)}=\bP^x(\fT \in dt)\bt{/dt}$ in $C^\infty$ such that all their derivatives are of exponential order. If, for some $x \in S$, the function $p^x_T(t)$ is rapidly decreasing at $0^+$ and $f(\lambda)$ varies regularly at infinity with index $\alpha>0$, then $p^x_{\fT}$ is rapidly decreasing at $0^+$.
\end{thm}
\begin{proof}
Let us define $\widetilde{F}(\lambda)=\cL[p^x_T](\lambda)$ and $\widetilde{\fF}(\lambda)=\cL[p^x_{\fT}](\lambda)$. Observe that they coincide with the Laplace-Stieltjes transforms of $\bP^x(T \le t)$ and $\bP^x(\fT \le t)$, so they are also defined by Eqs. \eqref{eq:Uforzero} and \eqref{eq:fUforzero}. Moreover, by Lemma \ref{lem:Taubrapid}, we know that $\widetilde{F}$ is such that for any $\alpha>0$ we have $\lim_{\lambda \to +\infty}\lambda^\alpha \widetilde{F}(\lambda)=0$.\\
From Eq. \eqref{eq:LaptransfFeqLaptransFoff} we know that $\widetilde{\fF}(\lambda)=\widetilde{F}(f(\lambda))$. Since $f$ is regularly varying at infinity with index $\alpha>0$, then there exists a slowly varying function $l(\lambda)$ such that 
\begin{equation}
f(\lambda)=\lambda^\alpha l(\lambda).
\end{equation}
By definition of slowly varying function at $\infty$, it is easy to see that for any $\beta>0$ $l^\beta(\lambda)$ is still a slowly varying function at $\infty$. Thus we know that $f^\beta(\lambda)$ is a regularly varying function with index $\alpha\beta>0$ \bt{by an application of \cite[Proposition 1.5.7]{bingham1989regular}}.\\
Fix now $k>0$ and observe that
\begin{equation}
\lambda^k\widetilde{\fF}(\lambda)=\lambda^k\widetilde{F}(f(\lambda)).
\end{equation}
Fix now $\beta>0$ such that $\alpha\beta>k$. Then
\begin{equation}
\lambda^k\widetilde{F}(f(\lambda))=\frac{\lambda^k}{f^\beta(\lambda)}f^\beta(\lambda)\widetilde{F}(f(\lambda))=\frac{1}{\lambda^{\beta \alpha-k}l^\beta(\lambda)}f^\beta(\lambda)\widetilde{F}(f(\lambda)).
\end{equation}
But we know that
\begin{equation}
\lim_{\lambda \to \infty}f^\beta(\lambda)\widetilde{F}(f(\lambda))=0
\end{equation}
and
\begin{equation}
\lim_{\lambda \to \infty}\lambda^{\beta\alpha-k}l^\beta(\lambda)=\infty
\end{equation}
so we have
\begin{equation}
\lim_{\lambda \to \infty}\lambda^k \widetilde{\fF}(\lambda)=\lim_{\lambda \to \infty}\frac{1}{\lambda^{\beta \alpha-k}l^\beta(\lambda)}f^\beta(\lambda)\widetilde{F}(f(\lambda))=0.
\end{equation}
We have shown that for any $k>0$ we have $\lim_{\lambda \to \infty}\lambda^k\widetilde{\fF}(\lambda)=0$, thus, by Lemma \ref{lem:Taubrapid}, we obtain that $p^x_{\fT}$ is rapidly decreasing at $0^+$.
\end{proof}
\section{Finite mean conditions for first passage times of Gauss-Markov processes}\label{sec:finmean}
\bt{Starting from \cite{sacric} (and later, e.g. \cite{maas, salinas}) Gauss-Markov processes have been frequently proposed to represent the membrane potential of a neuron in LIF models and systematic theoretical and computational studies on the first passage time through a threshold have been conducted (e.g. \cite{bensaczuc, herr, sactamzuc}). Hence we derive in this section some conditions on Gauss-Markov processes in order to apply the results in the previous sections.}
\mg{Since some of the proofs of this section are cumbersome, the latter will be shown in Appendix \ref{app:A}}
\textcolor{black}{\subsection{Gauss-Markov processes}
Following the lines of \cite{mehr} let us introduce the class of Gauss-Markov processes. Let us consider a Gaussian process $\{G(t), \  t \in [a,b]\}$ for $[a,b]\subset \R$ such that
\begin{itemize}
\item The sample paths of $G(t)$ are continuous almost surely;
\item $m_G(t):=\E[G(t)]$ is a continuous function in $[a,b]$;
\item $c_G(\tau,t):=Cov(G(\tau),G(t))$ is a continuous function in $[a,b]^2$;
\item $G(t)$ is non-degenerate except at most in the end-points $a,b$.
\end{itemize}
Moreover we say that the covariance $c_G(\tau,t)$ is triangular if there exist two continuous functions $u_G$ and $v_G$ on $[a,b]$ such that, whenever $\tau \le t$, $c_G(\tau,t)=u_G(\tau)v_G(t)$. One can show the following Proposition (see \cite[Theorem $1$]{mehr})
\begin{prop}
$G$ is a Markov process if and only if $c_G$ is triangular.
\end{prop}
We call such processes Gauss-Markov processes. Moreover, we call ratio function of $G$ the function $r_G(t)=u_G(t)/v_G(t)$. For such function one can show the following Proposition (see \cite[Remark $2$]{mehr})
\begin{prop}
The function $r_G(t)$ is continuous and strictly increasing.
\end{prop}
Since $r_G$ is monotone, it is almost everywhere differentiable. However, in the following, it will be useful to suppose that $r_G \in C^1([a,b])$.}
\subsection{Transformations of Gauss-Markov processes}
Transformations of Gauss-Markov processes have been very useful to determine some properties of first passage times of such processes through some fixed thresholds, making them derive from known properties of first passage times of other processes such as Wiener process or Ornstein-Uhlenbeck process. The first big result in such context is Doob's Transformation Theorem \cite{doob1949heuristic} which states:
\begin{thm}[\textbf{Doob's Transformation Theorem}]\label{thm:Doobtrans}
Let $\{G(t), t \ge t_0\}$ be a Gauss-Markov process with mean $m_G(t)$, covariance $c_G(\tau,t)=u_G(\tau)v_G(t)$ with $\tau\le t$ and ratio $r_G(t)=\frac{u_G(t)}{v_G(t)}$. Suppose $G(t_0)=m_G(t_0)$ almost surely and consider a standard Wiener process $W(t)$. Define
\begin{align}\label{eq:Wienertschange}
\rho_{G,W}(t)=\kappa r_G(t), && \varphi_{G,W}(t)=\frac{v_G(t)}{\sqrt{\kappa}}
\end{align}
for an arbitrary constant $\kappa>0$. Then
\begin{equation}\label{eq:Doobtrans}
G(t)=m_G(t)+\varphi_{G,W}(t)W(\rho_{G,W}(t)).
\end{equation}
\end{thm}
The constant $\kappa>0$ plays the role of a dimensional constant which can be useful for modelling purposes. In \bt{\cite{buonocore2011first} we find} another transformation theorem, this time with respect to an Ornstein-Uhlenbeck process:
\begin{thm}\label{thm:OUtrans}
Let $\{G(t), t \ge t_0\}$ be a Gauss-Markov process with mean $m_G(t)$, covariance $c_G(\tau,t)=u_G(\tau)v_G(t)$ with $\tau\le t$ and ratio $r_G(t)=\frac{u_G(t)}{v_G(t)}$. Suppose $G(t_0)=m_G(t_0)$ almost surely and consider an Ornstein-Uhlenbeck process $U(t)$ solution of
\begin{equation*}
dU_t=-\frac{1}{\theta}U_tdt+\sigma dW_t, \ U_0=0.
\end{equation*}
Define
\begin{align}\label{eq:OUtschange}
\rho_{G,U}(t)=\frac{\theta}{2}\ln\left(1+\frac{2\kappa}{\theta}r_G(t)\right), && \varphi_{G,U}(t)=\frac{v_G(t)}{\sigma\sqrt{\kappa}}\sqrt{1+\frac{2\kappa}{\theta}r_G(t)}
\end{align}
for an arbitrary constant $\kappa>0$. Then
\begin{equation}\label{eq:OUtrans}
G(t)=m_G(t)+\varphi_{G,U}(t)U(\rho_{G,U}(t)).
\end{equation}
\end{thm}
Here we propose a more general transformation theorem which involves just two Gauss-Markov processes:
\begin{thm}\label{thm:GMtrans}
Let $\{G_i(t), t \ge 0\}$ be Gauss-Markov processes for $i=1,2$ respectively with mean $m_{G_i}(t)$, covariance $c_{G_i}(\tau,t)=u_{G_i}(\tau)v_{G_i}(t)$ with $\tau\le t$ and ratio $r_{G_i}(t)=\frac{u_{G_i}(t)}{v_{G_i}(t)}$ whose derivative $\dot{r}_{G_i}(t)\not = 0$ for all $t \ge 0$. Suppose $G_{i}(0)=m_{G_i}(0)$ almost surely and define
\begin{align}\label{eq:GMtschange}
\rho_{G_1,G_2}(t)=r_{G_2}^{-1}(r_{G_1}(t)), && \varphi_{G_1,G_2}(t)=\frac{v_{G_1}(t)}{v_{G_2}(\rho_{G_1,G_2}(t))}.
\end{align}
Then
\begin{equation}\label{eq:GMtrans}
G_1(t)=m_{G_1}(t)-\varphi_{G_1,G_2}(t)m_{G_2}(\rho_{G_1,G_2}(t))+\varphi_{G_1,G_2}(t)G_2(\rho_{G_1,G_2}(t))
\end{equation}
in \mg{$1$-dimensional distributions}.
\end{thm}
\begin{rmk}
One can derive Theorem \ref{thm:OUtrans} from Theorem \ref{thm:GMtrans}. Indeed one can consider $G_1(t)$ as the GM process $G(t)$ and $G_2(t)$ as the Ornstein-Uhlenbeck process $U(t)$ \mg{and $t_0=0$}. In such case we have $m_{U}(t)=0$ and
\begin{align*}
u_U(t)=\frac{\sigma \theta}{2}\left(e^{\frac{t}{\theta}}-e^{-\frac{t}{\theta}}\right), && v_U(t)=\sigma e^{-\frac{t}{\theta}},
\end{align*}
obtaining the ratio
\begin{equation*}
r_U(t)=\frac{\theta}{2}\left(e^{\frac{2t}{\theta}}-1\right)
\end{equation*}
with inverse
\begin{equation*}
r_U^{-1}(t)=\frac{\theta}{2}\ln\left(1+\frac{2}{\theta}t\right).
\end{equation*}
Thus, by using the definition in Theorem \ref{thm:GMtrans}, we obtain
\begin{equation}\label{eq:esOU1}
\rho_{G,U}(t)=r_U^{-1}(r_G(t))=\frac{\theta}{2}\ln\left(1+\frac{2}{\theta}r_G(t)\right)
\end{equation}
which is the same function as in Theorem \ref{thm:OUtrans} for $\kappa=1$. Moreover we have
\begin{equation*}
v_U(\rho_{G,U}(t))=\frac{\sigma}{\sqrt{1+\frac{2}{\theta}r_G(t)}}
\end{equation*}
and then, by still using the definition in Theorem \ref{thm:GMtrans}, we obtain
\begin{equation}\label{eq:esOU2}
\varphi_{G,U}(t)=\frac{v_G(t)}{v_U(\rho_{G,U}(t))}=\frac{v_G(t)}{\sigma}\sqrt{1+\frac{2}{\theta}r_G(t)}
\end{equation}
which is the same function as in Theorem \ref{thm:OUtrans} for $\kappa=1$. Finally, substituting Eq. \eqref{eq:esOU1}, \eqref{eq:esOU2} and $m_U(t)=0$ in Eq. \eqref{eq:GMtrans} we obtain Eq. \eqref{eq:OUtrans}.
\end{rmk}
\subsection{First passage time densities and transformation formulas}
As one wants to study the first passage time density of a GM process $G(t)$ through a $C^2$ threshold $S_G(t)$, one can use transformation formulas to connect such density with other first passage time densities. A well known result in such direction is given in \cite{di2001computational}.
\begin{prop}\label{prop:fptWiener}
Let $\{G(t), t\ge 0\}$ be a GM process with mean $m_G(t)$, covariance $c_G(\tau,t)=u_G(\tau)v_G(t)$ for $\tau\le t$ and ratio $r_G(t)$. Let also $S_G(t)$ be any $C^2([0,+\infty[)$ function and
\begin{equation*}
T_G=\inf\{t \ge 0: \ G(t)>S_G(t)\}
\end{equation*}
with density $f_G(t)$. Consider $W(t)$ a standard Wiener process and pose
\begin{equation*}
S_W(t)=\frac{S_G(r_G^{-1}(t))-m_G(r_G^{-1}(t))}{v_G(r_G^{-1}(t))}
\end{equation*}
and
\begin{equation*}
T_W=\inf\{t \ge 0: \ W(t)>S_W(t)\}
\end{equation*}
with density $f_W(t)$. Then
\begin{equation}\label{eq:fptWiener}
f_G(t)=\dot{r}_G(t)f_W(r_G(t)).
\end{equation}
\end{prop}
In \cite{buonocore2011first} an analogue result, deriving from Theorem \ref{thm:OUtrans}, is shown.
\begin{prop}\label{thm:OUFPT}
Let $\{G(t), t\ge 0\}$ be a GM process with mean $m_G(t)$, covariance $c_G(\tau,t)=u_G(\tau)v_G(t)$ for $\tau\le t$ and ratio $r_G(t)$. Let also $S_G(t)$ be any $C^2([0,+\infty[)$ function and
\begin{equation*}
T_G=\inf\{t \ge 0: \ G(t)>S_G(t)\}
\end{equation*}
with density $f_G(t)$. Consider $U(t)$ an Ornstein-Uhlenbeck process as in Theorem \ref{thm:OUtrans} and pose
\begin{equation*}
S_U(t)=\frac{S_G(\rho_{G,U}^{-1}(t))-m_G(\rho_{G,U}^{-1}(t))}{\varphi_{G,U}(\rho_{G,U}^{-1}(t))}
\end{equation*}
where $\rho_{G,U}$ and $\varphi_{G,U}$ are defined in Theorem \ref{thm:OUtrans} and
\begin{equation*}
T_U=\inf\{t \ge 0: \ U(t)>S_U(t)\}
\end{equation*}
with density $f_U(t)$. Then
\begin{equation}\label{eq:fptOU}
f_G(t)=\dot{\rho}_{G,U}(t)f_U(\rho_{G,U}(t)).
\end{equation}
\end{prop}
Let us show a more general result.
\begin{prop}\label{prop:fpttrans}
Let $\{G_i(t), t\ge 0\}$ for $i=1,2$ be GM processes with mean $m_{G_i}(t)$, covariance $c_{G_i}(\tau,t)=u_{G_i}(\tau)v_{G_i}(t)$ for $\tau\le t$ and ratio $r_{G_i}(t)$. Let also $S_{G_1}(t)$ be any $C^2([0,+\infty[)$ function and
\begin{equation*}
T_{G_1}=\inf\{t \ge 0: \ G_1(t)>S_{G_1}(t)\}
\end{equation*}
with density $f_{G_1}(t)$. Pose
\begin{equation*}
S_{G_2}(t)=\frac{S_{G_1}(\rho_{G_1,G_2}^{-1}(t))-m_{G_1}(\rho_{G_1,G_2}^{-1}(t))}{\varphi_{G_1,G_2}(\rho_{G_1,G_2}^{-1}(t))}+m_{G_2}(t)
\end{equation*}
where $\rho_{G_1,G_2}$ and $\varphi_{G_1,G_2}$ are defined in Theorem \ref{thm:GMtrans} and
\begin{equation*}
T_{G_2}=\inf\{t \ge 0: \ G_2(t)>S_{G_2}(t)\}
\end{equation*}
with density $f_{G_2}(t)$. Then
\begin{equation}\label{eq:fptGM}
f_{G_1}(t)=\dot{\rho}_{G_1,G_2}(t)f_{G_2}(\rho_{G_1,G_2}(t)).
\end{equation}
\end{prop}
\begin{proof}
By Theorem \ref{thm:GMtrans} we know that
\begin{equation}\label{eq:pass1fpt}
G_1(t)=m_{G_1}(t)-\varphi_{G_1,G_2}(t)m_{G_2}(\rho_{G_1,G_2}(t))+\varphi_{G_1,G_2}(t)G_2(\rho_{G_1,G_2}(t))
\end{equation}
in \mg{$1$-dimensional distributions}. Consider the distribution functions $F_{T_i}(t)$ of $T_i$ for $i=1,2$. Thus we have that:
\begin{equation*}
F_1(t)=\bP(T_1 \le t)=\bP\left(\left\{\exists\tau\le t: \ G_1(\tau)>S_{G_1}(\tau)\right\}\right)
\end{equation*}
and then, by using Eq. \eqref{eq:pass1fpt}
\begin{equation*}
F_1(t)=\bP\left(\left\{\exists\tau\le t: \ G_2(\rho_{G_1,G_2}(\tau))>\frac{S_{G_1}(\tau)-m_{G_1}(\tau)}{\varphi_{G_1,G_2}(\tau)}+m_{G_2}(\rho_{G_1,G_2}(\tau))\right\}\right)
\end{equation*}
that is, by definition of $S_{G_2}(t)$
\begin{equation*}
F_1(t)=\bP\left(\left\{\exists\tau \le t: \ G_2(\rho_{G_1,G_2}(\tau))>S_{G_2}(\rho_{G_1,G_2}(\tau))\right\}\right).
\end{equation*}
Let us remark that as $r_{G_i}$ is continuous and increasing for $i=1,2$, also $r_{G_2}^{-1}$ is continuous and increasing and then $\rho_{G_1,G_2}$ is a continuous increasing function. By the intermediate value theorem we can write
\begin{equation}\label{eq:pass2fpt}
F_1(t)=\bP\left(\left\{\exists\theta \le \rho_{G_1,G_2}(t): \ G_2(\theta)>S_{G_2}(\theta)\right\}\right)=F_2(\rho_{G_1,G_2}(t)).
\end{equation}
Finally by deriving Eq. \eqref{eq:pass2fpt} we obtain \eqref{eq:fptGM}.
\end{proof}
\subsection{\bt{Deducing finite mean conditions}}
Our final aim in this section is to deduce some finite mean conditions for GM processes by using other GM processes for which such conditions are known. Let us give a criterion in such direction.
\begin{prop}\label{prop:finmeantrans}
Consider $\{G_i(t),t \ge 0\}$ for $i=1,2$ as in Proposition \ref{prop:fpttrans}. Suppose that there exists a constant $k\ge 0$ such that:
\begin{equation}\label{eq:suplincond}
\inf_{[k,+\infty[}\dot{\rho}_{G_1,G_2}(t)=\alpha>0.
\end{equation}
Then, if $\E[T_2]<+\infty$, we have $\E[T_1]<+\infty$.
\end{prop}
\begin{proof}
Let us first study some implications of the condition in Eq. \eqref{eq:suplincond}. For $t \ge k$ we have $\dot{\rho}_{G_1,G_2}(t)\ge \alpha$ so we have
\begin{equation*}
\rho_{G_1,G_2}(t)-\rho_{G_1,G_2}(k)\ge \alpha(t-k).
\end{equation*}
Posing $c=\rho_{G_1,G_2}(k)-\alpha k$ we have
\begin{equation*}
\rho_{G_1,G_2}(t)\ge \alpha t+c.
\end{equation*}
Since $\rho_{G_1,G_2}$ is an increasing function, also $\rho_{G_1,G_2}^{-1}$ is an increasing function and then
\begin{equation*}
t\ge \rho_{G_1,G_2}^{-1}(\alpha t+c).
\end{equation*}
Let us pose $s=\alpha t+c$ to obtain $t=\frac{s-c}{\alpha}$ and then
\begin{equation}\label{eq:pass1finmeantrans}
\rho_{G_1,G_2}^{-1}(s)\le \frac{s-c}{\alpha}.
\end{equation}
Finally observe that $t>k$ if and only if $s>\rho_{G_1,G_2}(k)$, so we have that Eq. \eqref{eq:pass1finmeantrans} is true for any $s>\rho_{G_1,G_2}(k)$.\\
Consider
\begin{equation}\label{eq:pass2finmeantrans}
\E[T_1]=\int_0^{+\infty}tf_{G_1}(t)dt=\int_0^{k}tf_{G_1}(t)dt+\int_{k}^{+\infty}tf_{G_1}(t)dt.
\end{equation}
It is easy to see that
\begin{equation*}
\int_0^{k}tf_{G_1}(t)dt\le k\int_0^{k}f_{G_1}(t)dt\le k<+\infty
\end{equation*}
so we only have to bound the second integral on the right-hand-side of Eq. \eqref{eq:pass2finmeantrans}. To do that, let us use Eq. \eqref{eq:fptGM} to obtain
\begin{equation*}
\int_{k}^{+\infty}tf_{G_1}(t)dt=\int_{k}^{+\infty}t\dot{\rho}_{G_1,G_2}(t)f_{G_2}(\rho_{G_1,G_2}(t))dt.
\end{equation*}
Since $\rho_{G_1,G_2}$ is a $C^1$-diffeomorphism, we can use a change of variable formula posing $s=\rho_{G_1,G_2}(t)$ to obtain
\begin{equation*}
\int_{k}^{+\infty}tf_{G_1}(t)dt=\int_{\rho_{G_1,G_2}(k)}^{+\infty}\rho_{G_1,G_2}^{-1}(s)f_{G_2}(s)ds
\end{equation*}
and, by Eq. \eqref{eq:pass1finmeantrans} we have
\begin{equation}\label{eq:pass3finmeantrans}
\int_{k}^{+\infty}tf_{G_1}(t)dt\le \frac{1}{\alpha}\left[\int_{\rho_{G_1,G_2}(k)}^{+\infty}sf_{G_2}(s)ds-c\int_{\rho_{G_1,G_2}(k)}^{+\infty}f_{G_2}(s)\right].
\end{equation}
But we also have that
\begin{equation*}
\int_{\rho_{G_1,G_2}(k)}^{+\infty}f_{G_2}(s)ds\le 1
\end{equation*}
and
\begin{equation*}
\int_{\rho_{G_1,G_2}(k)}^{+\infty}sf_{G_2}(s)ds\le \int_{0}^{+\infty}sf_{G_2}(s)ds=\E[T_2]<+\infty
\end{equation*}
so, by Eq. \eqref{eq:pass3finmeantrans}, we finally obtain
\begin{equation*}
\int_{k}^{+\infty}tf_{G_1}(t)dt<+\infty.
\end{equation*}
\end{proof}
Thanks to this result, one has only to choose a suitable $G_2$ for which finiteness of the mean of the first passage time is already known. Let us recall a result given in \cite{giorno1990asymptotic} using the form of \cite[Claim 8]{buonocore2011first}.
\begin{prop}\label{prop:OUasymp}
Let $U(t)$ be an Ornstein-Uhlenbeck process defined as solution of
\begin{equation*}
dU_t=-\frac{1}{\theta}U_tdt+\sigma dW_t
\end{equation*}
where $W_t$ is a Wiener process and $\theta,\sigma>0$ are constants. Let $S_U(t)$ be a $C^2$ function such that $\lim_{t \to +\infty}S_U(t)=S_U>2\sigma\sqrt{\theta}$. Define
\begin{equation*}
h_U=\frac{S_U}{\sigma\sqrt{\pi \theta}}e^{-\frac{S_U^2}{\sigma^2\theta}}
\end{equation*}
and
\begin{equation*}
T=\inf\{t\ge 0: \ U(t)>S_U(t)\}
\end{equation*}
with density $f_T$. Then, as $t \to +\infty$
\begin{equation*}
f_T(t)\sim \frac{h_U}{\theta}e^{-\frac{h_U}{\theta}t}.
\end{equation*}
\end{prop}
By using such proposition one can show the following
\begin{cor}\label{cor:OUfinmean}
Let $U(t)$ be an Ornstein-Uhlenbeck process defined as solution of
\begin{equation*}
dU_t=-\frac{1}{\theta}U_tdt+\sigma dW_t, \ \mg{U_0=0}
\end{equation*}
where $W_t$ is a standard Wiener process and $\theta,\sigma>0$ are constants. Let $S_U(t)$ be an upper bounded function and
\begin{equation*}
T=\inf\{t\ge 0: \ U(t)>S_U(t)\}
\end{equation*}
with density $f_T$. Then $\E[T]<+\infty$.
\end{cor}
\begin{proof}
By hypothesis there is a constant $M>2\sigma\sqrt{\theta}$ such that \bt{ for any }$ t \ge 0$ \bt{ it is true that} $S_U(t)\le M$. Define
\begin{equation*}
\tilde{T}=\inf\{t\ge 0: \ U(t)>M\}
\end{equation*}
with density $f_{\tilde{T}}$. Let us show that $T \le \tilde{T}$ almost surely. Fix $\omega \in \Omega$ and observe that if $U(t,\omega)>M$ then $U(t,\omega)>S_U(t)$. Then we have that
\begin{equation*}
\{t \ge 0: \ U(t,\omega)>M\}\subseteq \{t \ge 0: \ U(t,\omega)>S_U(t)\}
\end{equation*}
and then, taking the infimum \mg{on the sets for any fixed} $\omega \in \Omega$ such that \mg{such} sets are non-empty, we obtain
\begin{equation*}
T(\omega)\le\tilde{T}(\omega).
\end{equation*}
Since this inequality is valid for almost all $\omega \in \Omega$, we also have
\begin{equation*}
\E[T]\le \E[\tilde{T}].
\end{equation*}
Now we only need to show that $\E[\tilde{T}]<+\infty$. Since $f_{\tilde{T}}$ is a density function, it is in $L^1([0,+\infty[)$, while the function $id(t)=t$ is in $L^{\infty}([0,k])$ for all $k>0$. Thus we have only to show that $t\mapsto tf_{\tilde{T}}(t)$ is integrable in a neighbourhood of $+\infty$. But it is trivial since, by using Proposition \ref{prop:OUasymp}, we have that for $t \to +\infty$, $tf_{\tilde{T}}(t)\sim \frac{h_U}{\theta}te^{-\frac{h_U}{\theta}t}$ which is integrable.
\end{proof}
Combining such result with Proposition \ref{prop:finmeantrans} we easily obtain the following
\begin{cor}\label{cor:FMCOU}
Let $G(t)$ be a GM process and $U(t)$ an Ornstein-Uhlenbeck process as in Theorem \ref{thm:OUtrans}. Let also $S_G(t)$ be a $C^2([0,+\infty[)$ function and
\begin{equation*}
T_G=\inf\{t \ge 0: \ G(t)>S_G(t)\}.
\end{equation*}
Pose
\begin{equation*}
S_U(t)=\frac{S_G(\rho_{G,U}^{-1}(t))-m_G(\rho_{G,U}^{-1}(t))}{\varphi_{G,U}(\rho_{G,U}^{-1}(t))}.
\end{equation*}
Then, under the hypotheses:
\begin{enumerate}
\item It exists a constant $k>0$ such that $\inf_{[k,+\infty[}\dot{\rho}_{G,U}(t)=\alpha> 0$;
\item $S_U(t)$ is upper bounded,
\end{enumerate}
we have $\E[T_G]<+\infty$.
\end{cor}
\begin{proof}
Define
\begin{equation*}
T_U=\inf\{t \ge 0: \ U(t)>S_U(t)\}
\end{equation*}
and observe that, by hypothesis $2$ and Corollary \ref{cor:OUfinmean}, $\E[T_U]<+\infty$. Finally, by hypothesis $1$, we can use Proposition \ref{prop:finmeantrans} to assure that $\E[T_G]<+\infty$.
\end{proof}

A suitable GM process to use for our purposes is the Wiener process with non-zero drift. Indeed we have
\begin{prop}
Let $W_d(t)=W(t)+dt$ be a Wiener process with positive drift $d>0$, $S_d(t)$ an upper-bounded continuous function with $S_d(0)>0$ and pose
\begin{equation*}
T_d=\inf\{t\ge 0: W_d(t)>S_d(t)\}.
\end{equation*}
Then $\E[T_d]<+\infty$.
\end{prop}
\begin{proof}
Let $M=\sup_{t\ge 0}S_d(t)>0$ and define
\begin{equation*}
\tilde{T}_d=\inf\{t\ge 0: W_d(t)>M\}
\end{equation*}
with density $f_{\tilde{T}_d}$. Let us first show that $T_d\le \tilde{T}_d$ almost surely. To do this, fix $\omega \in \Omega$ and observe that
\begin{equation*}
W_d(t,\omega)>M \Rightarrow W_d(t,\omega)>S_d(t)
\end{equation*}
so
\begin{equation*}
\{t \ge 0: W_d(t,\omega)>M\}\subseteq \{t \ge 0: W_d(t,\omega)>S_d(t)\}
\end{equation*}
thus, taking the infimum \mg{on the sets} when for $\omega \in \Omega$ \mg{such} sets are non-empty, we have
\begin{equation*}
T_d(\omega) \le \tilde{T}_d(\omega)
\end{equation*}
Since such relation is true for almost all $\omega \in \Omega$ we also have
\begin{equation*}
\E[T_d]\le \E[\tilde{T}_d]
\end{equation*}
and then we only need to show that $\E[\tilde{T}_d]<+\infty$. But it trivial since
\begin{equation*}
f_{\tilde{T}_d}(t)=\frac{M}{\sqrt{2\pi t^3}}e^{-\frac{(M-dt)^2}{2t}}
\end{equation*}
and then $tf_{\tilde{T}_d}(t)$ is integrable.
\end{proof}
\subsection{The asymptotic behaviour at zero}
From Doob's Transformation Theorem one can also obtain some results on the asymptotic behaviour of the distribution function of the first passage time of a Gauss-Markov process through a fixed $C^2$ threshold. The following result represents a first step in such direction:
\begin{prop}\label{prop:GMconfr0}
Let $\{G(t),t\ge 0\}$ a Gauss-Markov process with mean $m_G(t)$, covariance $c_G(t,\tau)=u_G(\tau)v_G(t)$ with $\tau \le t$ and ratio $r_G(t)=\frac{u_G(t)}{v_G(t)}$. Suppose $G(0)=m_G(0)$ almost surely. Let also $S_G(t)$ be any $C^2([0,+\infty[)$ function such that $S_G(0)>m_G(0)$ and:
\begin{equation*}
T_G=\inf \{t \ge 0: \ G(t)>S_G(t)\}
\end{equation*}
with distribution function $F_G(t)=\bP(T_G \le t)$. Thus there are five positive constants $\delta,C_1,C_2,D_1,D_2$ such that for any $t \in [0,\delta]$ we have:
\begin{equation}\label{eq:confr0}
C_1\int_0^{r_G(t)}\frac{e^{-\frac{D_1}{2}}}{s^{\frac{3}{2}}}ds \le F_G(t)\le
C_2\int_0^{r_G(t)}\frac{e^{-\frac{D_2}{2}}}{s^{\frac{3}{2}}}ds.
\end{equation}
\end{prop}
This result can be used to show that under some hypothesis on $r_G(t)$ the distribution function $F_G(t)$ does not vary regularly in $0$.\\
To do this, we need the following technical lemma:
\begin{lem}\label{lem:teclem}
Let $C>0$ and $\rho:[0,+\infty[\to[0,+\infty[$ be a strictly increasing and \bt{differentiable} (in $]0,+\infty[$) function such that:
\begin{itemize}
\item[R1] $\rho(0)=0$;
\item[R2] There exists a constant $l_1>0$ such that $$\lim_{t \to 0^+}\frac{\rho(t)}{t}=l_1;$$
\item[R3] There exists a constant $l_2$ such that $$\lim_{t \to 0^+}\frac{\rho(t)-l_1t}{t^2}=l_2.$$
\end{itemize}
Consider the function
\begin{equation}\label{eq:tecfun}
F(t)=\int_0^{\rho(t)}s^{-\frac{3}{2}}e^{-\frac{C}{s}}ds.
\end{equation}
Then, for some positive constants $K_1,K_2$, as $t \to 0^+$ we have
\begin{equation}\label{eq:tecconfr}
F(t)\sim K_1t^{\frac{1}{2}}e^{-\frac{K_2}{t}}
\end{equation}
\end{lem}
\begin{rmk}
Hypotheses R(1-3) can be achieved if $\rho$ is a strictly increasing $C^2([0,+\infty[)$ function with $\rho(0)=0$ and $l_1=\dot{\rho}(0)>0$. Hypotheses R1 and R2 are obviously achieved by such conditions. Moreover, if we consider the Taylor polynomial
\begin{equation*}
p_2(t)=t\dot{\rho}(0)+\frac{\ddot{\rho}(0)}{2}t^2=tl_1+\frac{\ddot{\rho}(0)}{2}t^2
\end{equation*}
we know that
\begin{equation*}
\lim_{t \to 0^+}\frac{\rho(t)-p_2(t)}{t^2}=0
\end{equation*}
that is to say
\begin{equation*}
\lim_{t \to 0^+}\frac{\rho(t)-l_1t}{t^2}-\frac{\ddot{\rho}(0)}{2}=0.
\end{equation*}
Thus we can pose $l_2=\frac{\ddot{\rho}(0)}{2}$ to obtain hypothesis R3.
\end{rmk}
The technical lemma we showed before allows us to prove the following:
\begin{prop}\label{prop:notregvar}
With the same notation and under the same hypotheses of Proposition \ref{prop:GMconfr0}, if $r_G(t)$ satisfies hypotheses R(1-3) of Lemma \ref{lem:teclem}, then $F_G$ does not vary regularly in $0$.
\end{prop}
\begin{proof}
From Proposition \ref{prop:GMconfr0} we know that there exists five constants $\delta,C_1,D_1,C_2,D_2$ such that for any $t \in [0,\delta]$ we have
\begin{equation*}
C_1\int_0^{r_G(t)}\frac{e^{-\frac{D_1}{s}}}{s^{\frac{3}{2}}}ds\le F_G(t)\le C_2\int_0^{r_G(t)}\frac{e^{-\frac{D_2}{s}}}{s^{\frac{3}{2}}}ds
\end{equation*}
Let us pose
\begin{equation*}
F_i(t)=\int_0^{r_G(t)}\frac{e^{-\frac{D_i}{s}}}{s^{\frac{3}{2}}}ds \ i=1,2
\end{equation*}
and observe that we can write for $t \in [0,\delta]$:
\begin{equation}\label{eq:nonvarpass1}
C_1F_1(t)\le F_G(t)\le C_2F_2(t).
\end{equation}
Fix now $a>1$ and observe that $\frac{\delta}{a}<\delta$, so that Eq. \eqref{eq:nonvarpass1} holds for any $t \in \left[0,\frac{\delta}{a}\right]$. Then for any $t \in \left[0,\frac{\delta}{a}\right]$ we also have
\begin{equation*}
C_1F_1(at)\le F_G(at)\le C_2F_2(at)
\end{equation*}
and then
\begin{equation*}
\frac{F_G(at)}{F_G(t)}\ge \frac{C_1}{C_2}\frac{F_1(at)}{F_2(t)}.
\end{equation*}
Since $r_G(t)$ satisfies hypotheses R(1-3), then by Lemma \ref{lem:teclem} we can find four constants $K_1^1,K_1^2,K_2^1,K_2^2$ such that posing:
\begin{equation*}
H_i(t)=K_1^it^{\frac{1}{2}}e^{-\frac{K_2}{t}} \ i=1,2
\end{equation*}
we have:
\begin{equation*}
\lim_{t \to 0^+}\frac{F_i(t)}{H_i(t)}=1.
\end{equation*}
We want to evaluate
\begin{equation}\label{eq:nonvarpass2}
\lim_{t \to 0^+}\frac{H_1(at)}{H_2(t)}=\lim_{t \to +\infty}\frac{K_1^1}{K_1^2}a^{\frac{1}{2}}e^{\frac{K_2^2}{t}-\frac{K_2^1}{at}}.
\end{equation}
Remarking that:
\begin{equation*}
\frac{K_2^2}{t}-\frac{K_2^1}{at}=\frac{1}{at}(aK_2^2-K_2^1)
\end{equation*}
one can choose $a>\max\left\{1,\frac{K_2^1}{K_2^2}\right\}$ to obtain
\begin{equation*}
\lim_{t \to 0^+}\frac{1}{at}(aK_2^2-K_2^1)=+\infty.
\end{equation*}
Using this result in Eq. \eqref{eq:nonvarpass2} we obtain that
\begin{equation*}
\lim_{t \to +\infty}\frac{H_1(at)}{H_2(t)}=+\infty.
\end{equation*}
Thus we can evaluate
\begin{equation*}
\lim_{t \to 0^+}\frac{C_1}{C_2}\frac{F_1(at)}{F_2(t)}=\lim_{t \to 0^+}\frac{C_1}{C_2}\frac{F_1(at)}{H_1(at)}\frac{H_2(t)}{F_2(t)}\frac{H_1(at)}{H_2(t)}=+\infty
\end{equation*}
and then by comparison
\begin{equation*}
\lim_{t \to 0^+}\frac{F_G(at)}{F_G(t)}=+\infty.
\end{equation*}
\end{proof}
Actually, we can show \mg{that $F_G(t)$ rapidly decays} at $0^+$.
\begin{prop}\label{prop:rapdec} 
Under the same hypotheses of Proposition \ref{prop:notregvar} $F_G(t)$ is rapidly decreasing at $0$.
\end{prop}
\begin{proof}
With the same notation as in Proposition \ref{prop:notregvar} let us consider the functions $F_1$ and $F_2$ such that for $t \in [0,\delta]$:
\begin{equation*}
C_1F_1(t)\le F_G(t) \le C_2 F_2(t)
\end{equation*}
and fix $\alpha>0$. Observe that
\begin{equation}\label{eq:rapiddecaypass1}
C_1\frac{F_1(t)}{t^\alpha}\le \frac{F_G(t)}{t^\alpha} \le C_2\frac{F_2(t)}{t^\alpha}
\end{equation}
and define $H_i$ for $i=1,2$ as in Proposition \ref{prop:rapdec}. Let us first observe that:
\begin{equation*}
\lim_{t \to 0^+}\frac{H_i(t)}{t^\alpha}=\lim_{t \to 0^+}K_1^it^{\frac{1}{2}-\alpha}e^{-\frac{K_2^i}{t}}=0
\end{equation*}
and then we have
\begin{equation*}
\lim_{t \to 0^+}\frac{F_i(t)}{t^\alpha}=\lim_{t \to 0^+}\frac{F_i(t)}{H_i(t)}\frac{H_i(t)}{t^\alpha}=0.
\end{equation*}
Finally, by using the comparison theorem in Eq. \eqref{eq:rapiddecaypass1} we have $\lim_{t \to 0^+}\frac{F_G(t)}{t^\alpha}=0$.
\end{proof}

\section{A Neuronal Model}
\label{secmodel}
In this section we focus on an application of the results in Section \ref{sec:main} and \ref{sec:finmean} to obtain a model for the membrane potential of a neuron such that its firing times have some particular properties. Let us recall the Leaky Integrate-and-Fire (LIF for short) model introduced by Lapique in 1907 (see \cite{abbott1999lapicque}) in its stochastic version (see, for instance, \cite{greenwood2016stochastic}). Denote with $V(t)$ the membrane potential of a neuron at time $t$, $\theta>0$ the characteristic time of the membrane, $\hat{V} \in \R$ the resting potential, $I(t)$ a function representing the external stimulus and $\sigma>0$ a positive constant. Then $V(t)$ solves the following Stochastic Differential Equation:
\begin{equation}
dV(t)=\left(-\frac{1}{\theta}(V(t)-\hat{V})+I(t)\right)dt+\sigma dW(t) \ t>0=:T_0.
\end{equation}
First let us observe that if $I(t)\equiv0$, then $\E^x[V(t)] \to \hat{V}$, hence the name \textit{resting potential}. Moreover, let us consider a reset condition. Suppose we restarted the process from a \textit{reset position} $V_{reset}$ at time $T_{n-1}$ for the $n-1$-th time and fix a threshold $V_{th}>V_{reset}$. Define
\begin{equation}
T_n:=\inf\{t \ge T_{n-1}: \ V(t)\ge V_{th}\} \ n \ge 1,
\end{equation}
where $T_0=0$. Then we pose $V(T_n-)=V_{th}$ and $V(T_n)=V_{reset}$ and we \textit{reset} the SDE. This random time $T_n$ is called $n$-th spike time and the random time $T_{n}-T_{n-1}$ is called inter-spike interval (ISI for short). By definition, ISIs are supposed to be independent and identical distributed, which is a common assumption (see, for instance, \cite{tuckwell2005introduction}). From now on, let us fix the initial datum $V(0)=V_0 \in \R$. An example of sample path of such process can be seen in Figure \ref{fig:trajneumod1} on the left.\\
Since $V_0$ is fixed and $I(t)$ is a deterministic function, the process $V(t)$ without the reset mechanism is a Gaussian process with mean
\begin{equation}
m_V(t)=(1-e^{-\frac{t}{\theta}})\hat{V}+e^{- \frac{t}{\theta}}V_0+e^{-\frac{t}{\theta}}\int_0^t e^{\frac{s}{\theta}}I(s)ds
\end{equation}
and covariance
\begin{equation}
c_V(\tau,t)=u_V(\tau)v_V(t)
\end{equation}
where
\begin{align}
u_V(t)=\frac{\sigma \theta}{2}(e^{\frac{t}{\theta}}-e^{-\frac{t}{\theta}}), && v_V(t)=\sigma e^{-\frac{t}{\theta}}
\end{align}
which is the same covariance of an Ornstein-Uhlenbeck process. In particular the ratio is given by
\begin{equation}\label{eq:rV}
r_V(t)=\frac{\theta}{2}\left(e^{\frac{2t}{\theta}}-1\right).
\end{equation}
If we consider an Ornstein-Uhlenbeck process as a solution of
\begin{equation}
dU(t)=-\frac{1}{\theta}U(t)dt+\sigma dW(t), \ U(0)=0
\end{equation}
then we have from Eq. \eqref{eq:OUtschange}
\begin{align}
\rho_{V,U}(t)=t, && \varphi_{V,U}(t)=1.
\end{align}
Moreover, if we pose
\begin{equation}
S_U(t)=V_{th}-(1-e^{-\frac{t}{\theta}})\hat{V}-e^{- \frac{t}{\theta}}V_0-e^{-\frac{t}{\theta}}\int_0^t e^{\frac{s}{\theta}}I(s)ds
\end{equation}
and define
\begin{align}
T_1&:=\inf\{t \ge 0: \ V(t)\ge V_{th}\} \\
T_U&:=\inf\{t \ge 0: \ U(t)\ge S_U(t)\}
\end{align}
respectively with probability density functions $f_{T_1}(t)$ and $f_{U}(t)$ we have by Proposition \ref{thm:OUFPT}
\begin{equation}
f_{T_1}(t)=f_{U}(t).
\end{equation}
Moreover, if we denote with $f_{ISI}(t)$ the probability density function of an ISI, if $V_0=\hat{V}=V_{reset}$, then $f_{T_1}(t)=f_{ISI}(t)$.\\
Finally, by Corollary \ref{cor:FMCOU}, we obtain that if there exists a constant $K \in \R$ such that
\begin{equation}\label{eq:stimuluscond}
e^{-\frac{t}{\theta}}\int_0^t e^{\frac{s}{\theta}}I(s)ds>K
\end{equation}
and $I(t)$ is a $C^1$ function then $\E[T_1]<+\infty$. Let us observe that such hypothesis is not unrealistic: indeed it is satisfied, for instance, by any constant or excitatory stimulus.\\
However, in \cite{gerstein1964random} it has been shown that the exponential-like behaviour of the tails of $T_1$ is not sufficient to describe the ISI distribution. In particular, the authors refer to the fact that stable distributions for the ISI could be much more realistic then exponential ones. Two of the main features that lead the authors to consider stable distributions, together with the invariance under affine transformation, are the fact that the ISIs seemed to have an heavy-tail behaviour and that such behaviour is confirmed by the fact that their sample mean does not converge. Thus, we will now propose a modification of the LIF model that produces heavy-tailed ISIs. The idea is to consider a time-changed LIF model, in order to produce semi-Markov dynamics for the membrane potential. Semi-Markov models for theoretical neuroscience are not unrealistic and have already been considered (see, for instance, [\cite{tuckwell2005introduction}, Section $10.10$]).\\
Let us consider an $\alpha$-stable subordinator $\sigma_\alpha(t)$ and its inverse $L_\alpha(t)$. Thus, let us define the process $V_\alpha(t):=V(L_\alpha(t))$ (an example of its sample path is given in Figure \ref{fig:trajneumod1} on the right) and denote with $\fT_\alpha$ the random variable that represents the duration of an ISI. In particular, let us suppose that $V_0=\hat{V}=V_{reset}$, so that the first passage time of the non-restarted process $V_\alpha(t)$ represents such random variable. Thus, if condition \ref{eq:stimuluscond} is satisfied, since $C:=\E[T_1]<+\infty$, we have, by Corollary \ref{cor:corasymbeh}, that
\begin{equation}
\bP(\fT_\alpha>t)\sim \frac{C}{t^{\alpha}\Gamma(1-\alpha)} \mbox{ as }t \to +\infty
\end{equation}
so that the ISIs show an heavy-tailed behaviour. Moreover, recalling that $r_V(t)$ is given in Eq. \eqref{eq:rV} and it is a $C^2$ function such that $r_V(0)=0$ and $m_V(0)=\hat{V}<V_{th}$, then we have, from Proposition \ref{prop:rapdec}, that the probability density function $f_{T_1}$ of $T_1$ is rapidly decreasing at $0^+$. Now, it is easy to see that since the Levy measure of a stable subordinator of exponent $\alpha$ is given by $\nu_\alpha(dy)=y^{-\alpha-1}dy$, if $\bP(T_1 \le t)$ is infinitely differentiable, then, by Proposition \ref{prop:infdiff}, we know that $\bP(\fT_\alpha \le t)$ is infinitely differentiable. Moreover, if all the derivatives of $\bP(T_1 \le t)$ and $\bP(\fT_\alpha \le t)$ are of exponential order, then, by Theorem \ref{thm:rapiddecayzero}, also the probability density function $f_{\fT}$ of $\fT$ is rapidly decreasing at $0^+$. This is a physiological acceptable property, since we do not expect the neuron to fire almost instantaneously. This behaviour is evident in Figure \ref{fig:neumod}. In particular on the left one can see the different tails of $\fT_\alpha$ for different values of $\alpha$, while on the right one can see a comparison with the tails of $T_1$.
\begin{figure}[h]
\centering
\includegraphics[width=0.49\linewidth]{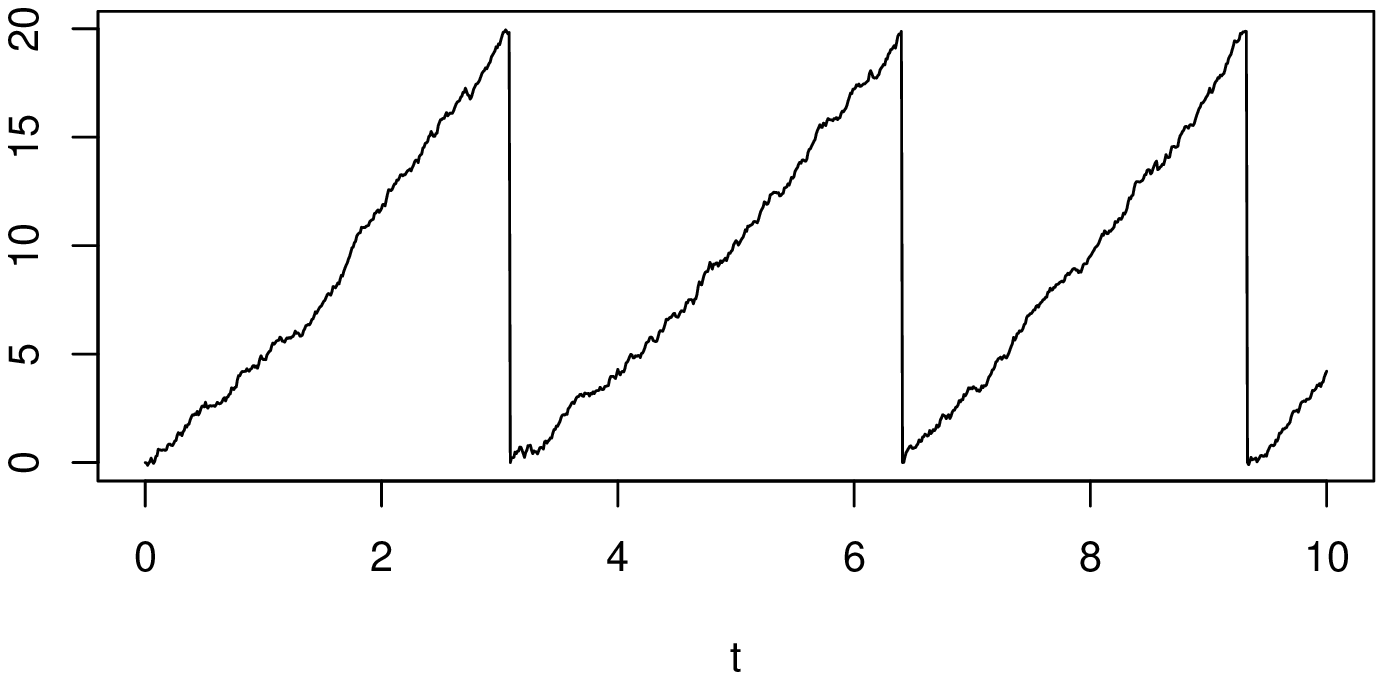}
\includegraphics[width=0.49\linewidth]{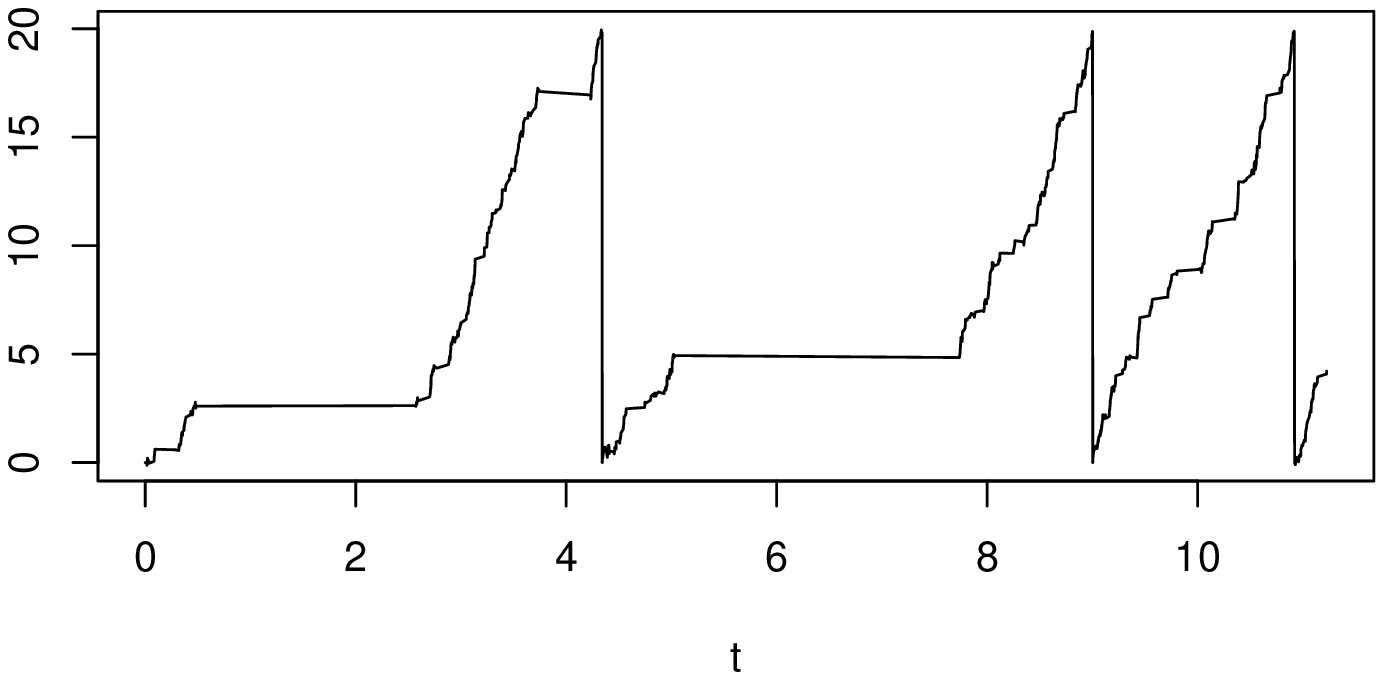}
\caption{Simulation of the neuronal model. On the left, a sample path of $V(t)$. On the right, the respective sample of $V_\alpha(t):=V(L(t))$. In particular, we setted $I(t)\equiv I_0=6$, $\sigma=1$, $V_0=\hat{V}=V_{reset}=0$, $V_{th}=20$, $\alpha=0.75$. Time steps for the simulation are $\Delta t=0.01$ and $\Delta y=0.01$.}
\label{fig:trajneumod1}
\end{figure}
\begin{figure}[h]
\centering
\includegraphics[width=0.49\linewidth]{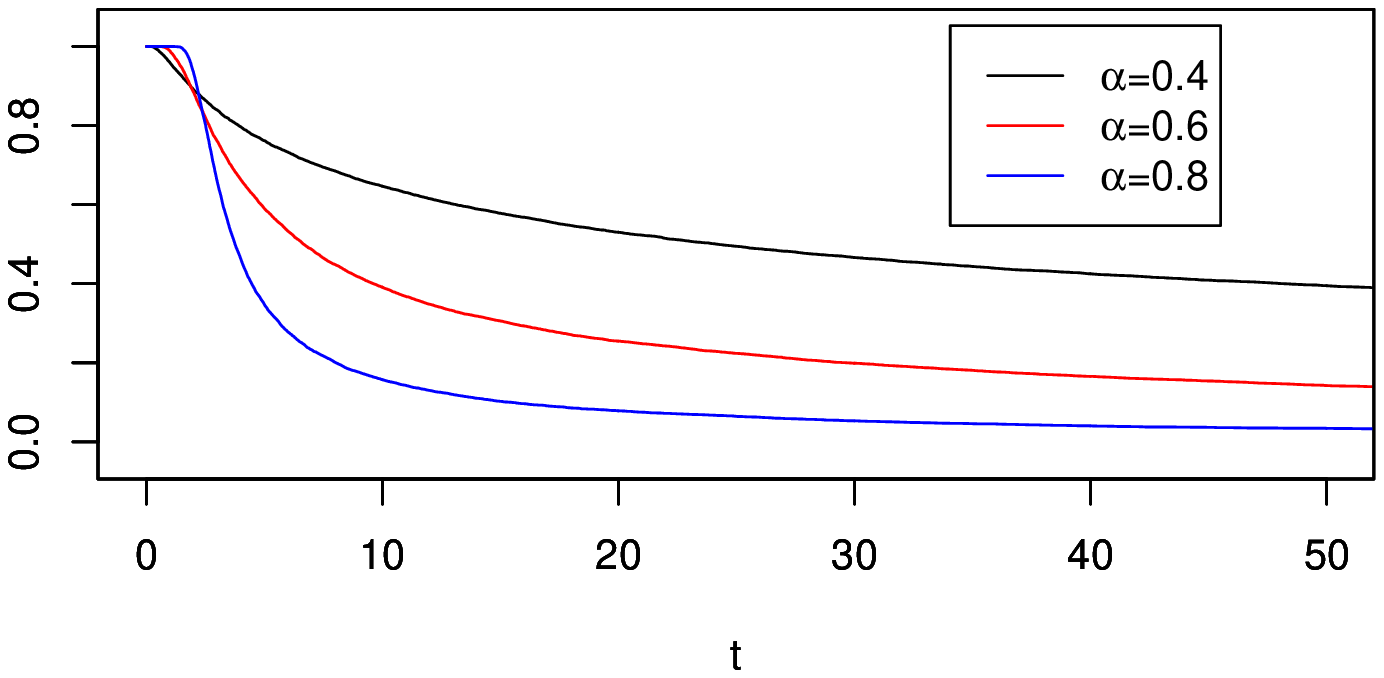}
\includegraphics[width=0.49\linewidth]{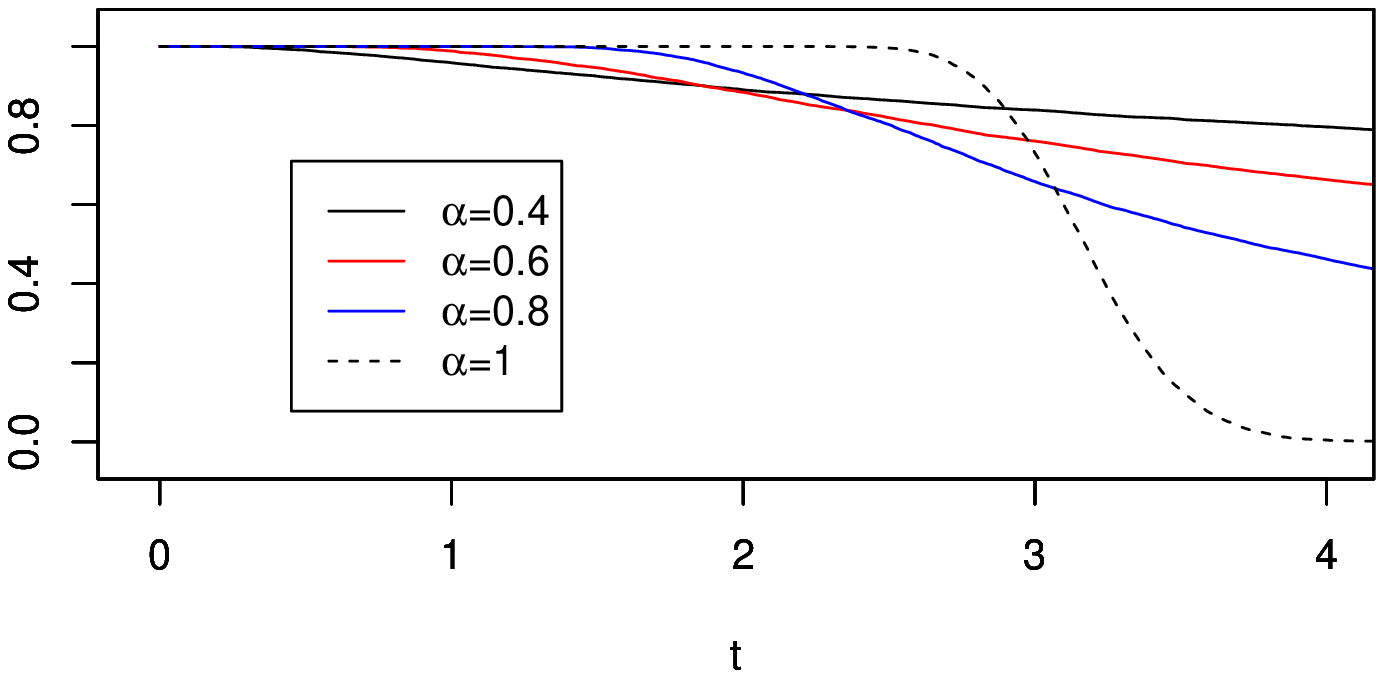}
\caption{Simulation of the neuronal model. On the left the function $\bP(\fT_\alpha >t)$ for different values of $\alpha$. On the right, the same plot zoomed in $[0,4]$, where the dashed line is the plot of the simulated function $\bP(T_1>t)$. We fixed $I(t)\equiv I_0=6$, $\sigma=1$, $V_0=\hat{V}=V_{reset}=0$ and $V_{th}=20$. Time steps for the simulation are $\Delta t=0.01$ and $\Delta y=0.01$ and $\bP(\fT_\alpha>t)$ is estimated by simulating $10000$ trajectories.}
\label{fig:neumod}
\end{figure}
\textcolor{black}{One could also take into account the process $N(t)$, which is the number of spikes of the neuron up to the time $t$ before the time change. It is a renewal process whose inter-jump times are i.i.d. random variables distributed as the first passage time $T$ of $V(t)$ through the threshold $V_{th}$. It is well known (see for instance \cite{buonocore2011first}) that if the stimulus is constant, $\bP(T<t)$ asymptotically behaves as an exponential, hence, for great jumps, $N(t)$ is similar to a Poisson process $P(t)$. If we consider the time changed process $V_\alpha(t)$ with its counting process $N_\alpha(t)$, then we can observe that $N_\alpha(t)=N(L_\alpha(t))$. Moreover, by using Proposition \ref{eq:exptoML}, we know that the inter-jump times $\fT$ are such that $\bP(\fT<t)$ asymptotically behaves as a Mittag-Leffler. Hence we could ask if we can approximate the process $N_\alpha(t)$ with a fractional Poisson process $P_\alpha(t)=P(L_\alpha(t))$. However, if we consider the asymptotic behavior at $0^+$ of $\bP(\fT<t)$, we have that, since $\bP(T<t)$ is rapidly decreasing at $0^+$ (by Prop. \ref{prop:rapdec}), also $\bP(\fT<t)$ is rapidly decreasing at $0^+$ (by Thm. \ref{thm:rapiddecayzero}) while the inter-jump times $J$ of a fractional Poisson process are such that $\bP(J<t)$ are regularly varying at $0^+$. Hence the approximation of the counting process $N_\alpha(t)$ with a fractional Poisson process works well for big values of the inter-jump times, while fails for small values of such times.}

\section{Simulation results}\label{sec:simul}
\bt{We provide in this section some techinques of stochastic simulation which may be used to verify the model. For thus we refer to Example \ref{ex:BMd}. Hence we first simulate} the process $W_\delta(t)=W(t)+\delta t$. It is well known (see, for instance, \cite{asmussen2007stochastic}) that such process (with initial datum $W_\delta(0)=0$) can be simulated by using a recursive scheme. Indeed, denoting with $\widetilde{W}_\delta$ the simulated process, if we consider a time step $\Delta t$, setting $t_n=n\Delta t$ for $n \in \N$, we have
\begin{equation}\label{eq:simBMdrift}
\begin{cases}
\widetilde{W}_\delta(0)=0\\
\widetilde{W}_\delta(t_n)=W_{\delta}(t_{n-1})+\delta \Delta t+\sqrt{\Delta t}Z_n & n \in \N
\end{cases}
\end{equation}
where $Z_n \sim \mathcal{N}(0,1)$ are independent and we pose
\begin{equation}
\widetilde{W}_\delta(t)=\widetilde{W}_\delta(t_{n-1}) \ \forall t \in [t_{n-1},t_n).
\end{equation}
To produce a time-changed Brownian motion with drift, we need then to simulate an inverse subordinator. Even in this case, if we can simulate a subordinator $\sigma$, then we can proceed with a recursive formula. Indeed, let us denote with $\widetilde{\sigma}$ and $\widetilde{L}$ respectively the simulated subordinator and the simulated inverse subordinator. Suppose $\widetilde{\sigma}$ has (discrete) state space $\widetilde{\Sigma}\subset [0,+\infty)$ and the time step of such process is $\Delta y$. Fix the time step for $\widetilde{L}$ as $\Delta t:=\min_{(x,y)\in \widetilde{\Sigma}^2}|x-y|$ and denote $y_m=m\Delta y$ for $m \in \N$ and $t_n=n\Delta t$ for $n \in \N$. Suppose we have simulated $\widetilde{L}(t_{n-1})$ and consider $M \in \N$ such that $y_M=\widetilde{L}(t_{n-1})$. Then we can simulate
\begin{equation}\label{eq:siminvsub}
\widetilde{L}(t_n)=\min\{y_m\ge y_M: \ \widetilde{\sigma}(y_m) \ge t_n\}.
\end{equation}
Now we need to establish how to simulate $\sigma$. First let us observe that for any $\Delta y$ we have $\sigma(y+\Delta y)-\sigma(y)\overset{d}{=}\sigma(\Delta y)$. Thus we have the recursive formula:
\begin{equation}
\begin{cases}
\widetilde{\sigma}(0)=0 \\
\widetilde{\sigma}(y_m)=\widetilde{\sigma}(y_{m-1})+\sigma(\Delta y).
\end{cases}
\end{equation}
\bt{Finally, we need simulate $\sigma(\Delta y)$. For this first fix a Laplace exponent $f$}. Thus we also know the Laplace transform of the variable $\sigma(\Delta y)$ given by $g(\lambda)=e^{-\Delta yf(\lambda)}$. Thus we have to simulate a random variable only knowing its Laplace transform. In such case, some simulation algorithms are given in \cite{devroye1981computer,devroye2013non} and compared in \cite{ridout2009generating}. Some of these methods require a numerical inversion of the Laplace transform, whose algorithms are discussed, for instance, in \cite{abate2000introduction}.\\
However, if $\sigma(t)$ is an $\alpha$-stable subordinator, one can use an ad-hoc simulation algorithm. In particular one has $\sigma(t)\overset{d}{=}t^{\frac{1}{\alpha}}\sigma(1)$, thus one has only to simulate a skew-symmetric $\alpha$-stable random variable $\sigma(1)$. For stable random variables $S \sim S(\alpha,\beta,\gamma,\delta;1)$ (here we use the notation in \cite{nolan2003stable}), one has a particular algorithm. First (see, for instance, \cite{asmussen2007stochastic}) for a variable $S \sim S(\alpha,0,1,0;1)$ we have that if $Y_1 \sim Exp(1)$ and $Y_2 \sim U\left(-\frac{\pi}{2},\frac{\pi}{2}\right)$, then
\begin{equation}
S\overset{d}{=}\frac{\sin(\alpha Y_2)}{(\cos(Y_2))^{\frac{1}{\alpha}}}\left(\frac{\cos((1-\alpha)Y_2)}{Y_1}\right)^{\frac{1-\alpha}{\alpha}}
\end{equation}
while for a general $S \sim S(\alpha,\beta,\gamma,\delta;1)$, if $S_1,S_2 \sim S(\alpha,0,1,0;1)$, then
\begin{equation}
S\overset{d}{=}\delta+\gamma\left(\frac{1+\beta}{2}\right)^{\frac{1}{\alpha}}S_1-\gamma\left(\frac{1-\beta}{2}\right)^{\frac{1}{\alpha}}S_2.
\end{equation}
To obtain a positive stable random variable (see, for instance, \cite{meerschaert2011stochastic}), we have to pose $S \sim S(\alpha,1,\gamma(\alpha),0;1)$ where
\begin{equation}
\gamma(\alpha)=\left(\cos\left(\frac{\pi\alpha}{2}\right)\right)^{\frac{1}{\alpha}}.
\end{equation}
However, to simulate stable random variables, we used the R package \textit{stabledist} (see \cite{wuertz6rmetrics}).\\
Thus, since we can simulate $W_\delta(t)$ and $L(t)$, we know how to simulate $X^f(t)=W_\delta(L(t))$ just by composing the simulation formulas (see, for instance, [\cite{meerschaert2011stochastic}, Example 5.21]). The same can be done for the standard Brownian motion by setting $\delta=0$.\\
For the first numerical experiment, we choose an $\alpha$-stable subordinator for $\alpha=0.7$, setted the drift coefficient $\delta=1$ and considered as open set $\mathcal{S}=(-\infty,1)$. We can see in Figure \ref{fig:stablealpha07bmd} on the left how the curves overlap. Denoting with $\fT_1$ the first exit time of $X^f_1(t):=W_1(L(t))$ from $\mathcal{S}$, since for the Brownian motion with drift we have that $\bP(\fT_1 >t)$ should have a power law decay, it could be interesting to study the convergence of 
$$RL_1(t):=\frac{\log(\bP(\fT_1>T))+\log(\Gamma(1-\alpha))}{\log(t)}$$
as $t \to +\infty$. Moreover, let us study also the convergence of 
$$R_1(t):=\frac{\bP(\fT_1>t)}{A_1(t)}$$
where 
$$A_1(t):=\frac{t^{-\alpha}}{\Gamma(1-\alpha)}.$$
In table \ref{tab:datasim} these values are shown for $t=25,50,75$: we can see that $RL_1(t)$ tends to $-0.75$ and $R_1(t)$ tends to $1$. For $t=75$, we have that only $110$ trajectories of our $10000$ simulated ones are such that $\fT_1>t$, so, since it is almost the $1\%$ of the trajectories, we can consider bigger values unreliable. The same numerical experiment has been repeated with $\delta=0$, obtaining the plot in figure \ref{fig:stablealpha07bmd} on the right. Denoting with $\fT_2$ the first exit time of $X^f_2(t):=W(L(t))$, let us consider the function $$R_2(t):=\frac{\bP(\fT_2>t)}{A_2(t)}$$
where 
$$A_2(t)=\frac{1}{\Gamma\left(1-\frac{\alpha}{2}\right)}[1-e^{-\sqrt{2t^{-\alpha}}}],$$
whose values for $t=25,50,75$ are shown in table \ref{tab:datasim}. Here, convergence is slower, since for $t=75$ we have $2050$ trajectories such that $\fT_2>t$, which is still a big number. We have also $R_2(100)=1.168963$ which is nearer to $1$, and for $t=100$ we have still $1812$ trajectories such that $\fT_2>t$.
\begin{figure}[h]
\centering
\includegraphics[width=0.49\linewidth]{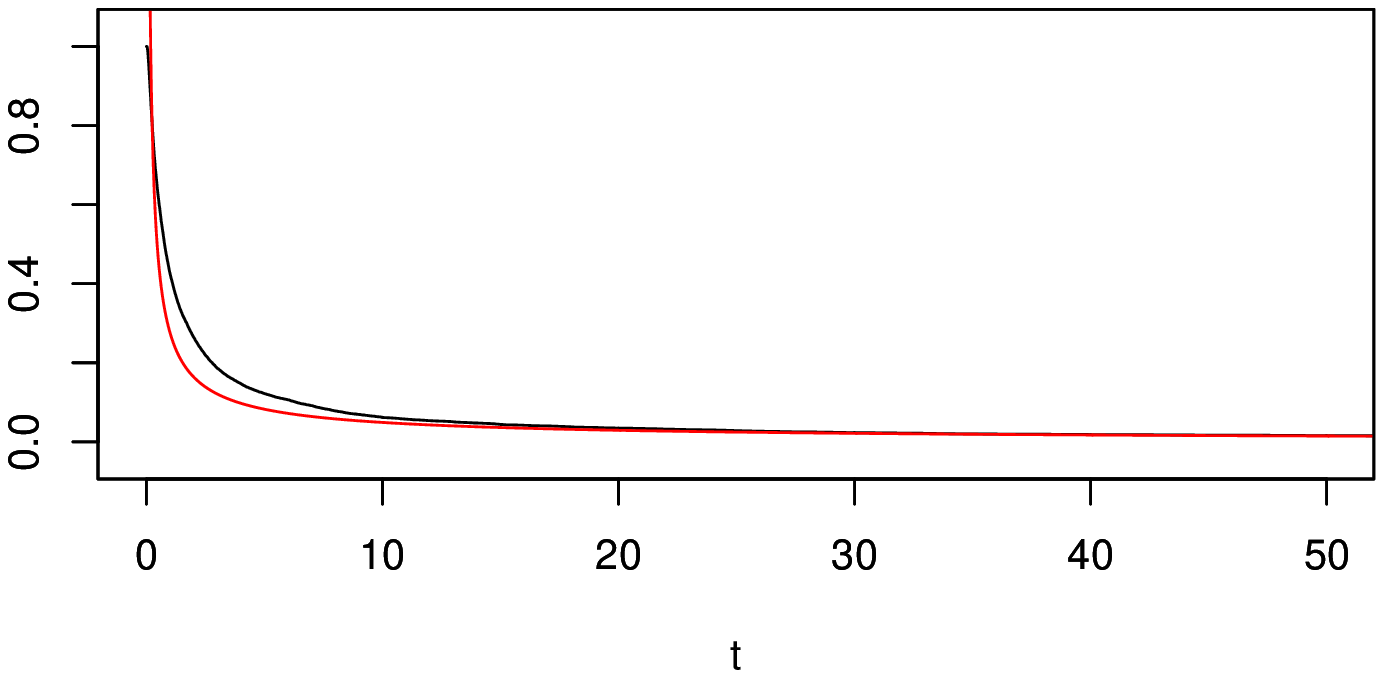}
\includegraphics[width=0.49\linewidth]{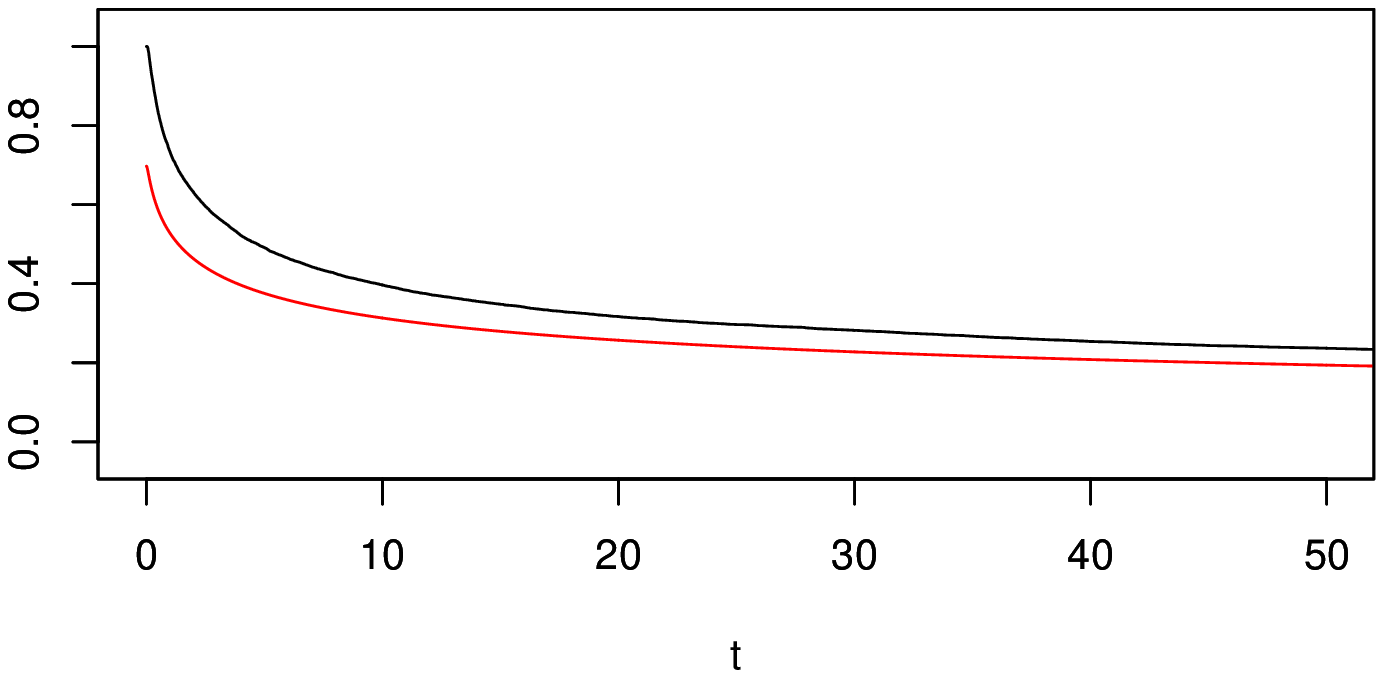}\caption{Numerical experiments. On the left: the curve in black is the plot of the simulated tail function $\bP(\fT_1 > t)$ for the first exit time $\fT_1$ of a time-changed Brownian motion with drift $X^f(t):=W_\delta(L(t))$ (where $L(t)$ is the inverse of an $\alpha$-stable subordinator) from an open set $\mathcal{S}=(-\infty,c)$, while the red line is the asymptotic estimate $\frac{c}{\delta t^\alpha \Gamma(1-\alpha)}$. On the right: the curve in black is the plot of the simulated tail function $\bP(\fT_2 > t)$ for the first exit time $\fT_2$ of a time-changed Brownian motion $X^f(t):=W(L(t))$ from the same open set $\mathcal{S}$, while the red line is the asymptotic estimate $\frac{1}{\Gamma\left(1-\frac{\alpha}{2}\right)}\left[1-e^{-c\sqrt{2t^{-\alpha}}}\right]$. In particular $c=1$, $\delta=1$ and $\alpha=0.75$. The simulation steps are $\Delta t=\Delta y=0.01$ and the estimate of $\bP(\fT_i>t), \ i=1,2$ has been done with $10000$ trajectories.}
\label{fig:stablealpha07bmd}
\end{figure}
\begin{table}[h]\label{tab:datasim}
\begin{tabular}{|l|l|l|l|}
\hline
$t$ & $25$         & $50$         & $75$         \\\hline
$RL_1(t)$               & $-0.7084494$ & $-0.7326317$ & $-0.7462296$ \\\hline
$R_1(t)$                & $1.14276$    & $1.070146$                       & $1.01631$    \\\hline
$R_2(t)$                & $1.233939$   & $1.220551$                       & $1.20361$   \\\hline
\end{tabular}
\caption{Values of the function $RL_1(t)$, $R_1(t)$ and $R_2(t)$.}
\end{table}
\appendix
\mg{\section{Proofs from Section \ref{sec:finmean}}\label{app:A}
\subsection{Proof of Theorem \ref{thm:GMtrans}}
By using Doob's Transformation Theorem there is a Wiener process $W(t)$ such that
\begin{equation}\label{eq:passtransteo0}
G_i(t)=m_{G_i}(t)+v_{G_i}(t)W(r_{G_i}(t)), \ i=1,2
\end{equation}
in law.
Then, considering the previous equation for $i=2$ we have
\begin{equation}\label{eq:passtransteo1}
W(r_{G_2}(t))=\frac{G_2(t)-m_{G_2}(t)}{v_{G_2}(t)}.
\end{equation}
Since $r_{G_2}(t)$ is continuous and strictly increasing, it is invertible; moreover, since $\dot{r}_{G_2}(t)\not = 0$ for all $t \ge 0$, $r_{G_2}^{-1}$ is a $C^1$ function. From equation \eqref{eq:passtransteo1} we have
\begin{equation*}
W(t)=\frac{G_2(r_{G_2}^{-1}(t))-m_{G_2}(r_{G_2}^{-1}(t))}{v_{G_2}(r_{G_2}^{-1}(t))}
\end{equation*}
and then, by definition of $\rho_{G_1,G_2}(t)$
\begin{equation}\label{eq:passtransteo2}
W(r_{G_1}(t))=\frac{G_2(\rho_{G_1,G_2}(t))-m_{G_2}(\rho_{G_1,G_2}(t))}{v_{G_2}(\rho_{G_1,G_2}(t))}.
\end{equation}
Finally, by substituting Eq. \eqref{eq:passtransteo2} in \eqref{eq:passtransteo0} for $i=1$ we obtain
\begin{equation}
G_1(t)=m_{G_1}(t)-\frac{v_{G_1}(t)}{v_{G_2}(\rho_{G_1,G_2}(t))}m_{G_2}(\rho_{G_1,G_2}(t))+\frac{v_{G_1}(t)}{v_{G_2}(\rho_{G_1,G_2}(t))}G_2(\rho_{G_1,G_2}(t))
\end{equation}
that is Eq. \eqref{eq:GMtrans} by definition of $\varphi_{G_1,G_2}(t)$.
\subsection{Proof of Proposition \ref{prop:GMconfr0}}
Let us suppose for simplicity $v_G(0)>0$. Consider a Wiener process $W(t)$ and define
\begin{equation*}
S_W(t)=\frac{S_G(t)-m_G(t)}{v_G(t)}
\end{equation*}
remarking that $S_W(0)>0$ and $S_W(t)$ is a continuous function in $[0,+\infty)$. Let us fix $\alpha$ such that $0<\alpha<S_W(0)$. Since $S_W$ is a continuous function there exists a $\delta>0$ such that $S_W(t)>\alpha>0$ for any $t \in [0,\delta]$. Now define:
\begin{equation*}
\widetilde{S}_W(t)=S_W(r_G^{-1}(t))
\end{equation*}
and
\begin{equation*}
T_W=\inf \{t>0 \ W(t)>S_W(t)\}
\end{equation*}
with probability density function $f_W(t)$ and distribution function $F_W(t)$. Consider $f_G(t)$ the probability density function of $T_G$. Thus by Proposition \ref{prop:fptWiener} we have
\begin{equation*}
f_G(t)=\dot{r}_G(t)f_W(r_G(t))
\end{equation*}
and then by integrating
\begin{equation*}
F_G(t)=\int_0^tf_G(s)ds=\int_0^t\dot{r}_G(s)f_W(r_G(s))ds.
\end{equation*}
By using the change of variable $z=r_G(s)$ we obtain:
\begin{equation}\label{eq:FGeqFW}
F_G(t)=\int_0^tf_G(s)ds=\int_0^{r_G(t)}f_W(z)dz=F_W(r_G(t)).
\end{equation}
Since $r_G$ is continuous and strictly increasing in $[0,\delta]$ then $r_G^{-1}$ is continuous (see for instance \cite{garling2014course}) and strictly increasing in $[0,r_G(\delta)]$. Thus we have that $\widetilde{S}_W(t)$ is a continuous function in $[0,r_G(\delta)]$. Moreover, since $r_G^{-1}(0)=0$, then $\widetilde{S}_W(0)=S_W(0)>0$ and, by definition of $\delta$, $\widetilde{S}_W(t)>\alpha>0$ for any $t \in [0,r_G(\delta)]$. Let us define
\begin{align*}
S_{min}=\min_{[0,r_G(\delta)]}\widetilde{S}_W(t) && 
S_{max}=\max_{[0,r_G(\delta)]}\widetilde{S}_W(t)
\end{align*}
and
\begin{align*}
T^{min}_W=\inf\{t>0: \ W(t)>S_{min}\} && 
T^{max}_W=\inf\{t>0: \ W(t)>S_{max}\}
\end{align*}
respectively with distribution functions $F_{min}(t)$ and $F_{max}(t)$. By definition of $S_{min}$ and $S_{max}$ we have 
\begin{equation*}
T^{min}_W \wedge r_G(\delta) \le T_W \wedge r_G(\delta) \le T^{max}_W \wedge r_G(\delta)
\end{equation*}
and thus, defining $\widetilde{F}_{min}(t)$, $\widetilde{F}_{W}(t)$, $\widetilde{F}_{max}(t)$ the distribution functions respectively of $T^{min}_W \wedge r_G(\delta)$, $T_W \wedge r_G(\delta)$ and $T^{max}_W \wedge r_G(\delta)$ we have
\begin{equation*}
\widetilde{F}_{max}(t)\le \widetilde{F}_{W}(t)\le \widetilde{F}_{min}(t).
\end{equation*}
For $t\le r_G(\delta)$ we have
\begin{equation*}
\widetilde{F}_{W}(t)=\bP(T_W \wedge r_G(\delta)\le t)=\bP(T_W \le t)=F_W(t)
\end{equation*}
and in a similar way we have $\widetilde{F}_{max}(t)=F_{max}(t)$ and $\widetilde{F}_{min}(t)=F_{min}(t)$. Thus we obtain for any $t \in [0,r_G(\delta)]$
\begin{equation*}
F_{max}(t)\le F_{W}(t)\le F_{min}(t).
\end{equation*}
For this reason we have for any $t \in [0,\delta]$
\begin{equation*}
F_{max}(r_G(t))\le F_{W}(r_G(t))\le F_{min}(r_G(t))
\end{equation*}
and then, by using Eq. \eqref{eq:FGeqFW}
\begin{equation*}
F_{max}(r_G(t))\le F_{G}(t)\le F_{min}(r_G(t)).
\end{equation*}
But since $\widetilde{S}_W(t)>\alpha>0$ for any $t \in [0,r_G(\delta)]$, $S_{max}\ge S_{min}>\alpha>0$ and then we have
\begin{align*}
F_{max}(r_G(t))=\frac{S_{max}}{\sqrt{2\pi}}\int_0^{r_G(t)}\frac{e^{-\frac{S^2_{max}}{2s}}}{s^{\frac{3}{2}}}ds && F_{min}(r_G(t))=\frac{S_{min}}{\sqrt{2\pi}}\int_0^{r_G(t)}\frac{e^{-\frac{S^2_{min}}{2s}}}{s^{\frac{3}{2}}}ds.
\end{align*}
Finally, posing:
\begin{align*}
C_1=\frac{S_{max}}{\sqrt{2\pi}} && D_1=\frac{S_{max}}{2}\\
C_2=\frac{S_{min}}{\sqrt{2\pi}} && D_2=\frac{S_{min}}{2}
\end{align*}
we obtain Eq. \eqref{eq:confr0}.
\subsection{Proof of Lemma \ref{lem:teclem}}
Let us remark that by definition $F$ is a differentiable function with derivative
\begin{equation*}
f(t)=\dot{\rho}(t)(\rho(t))^{-\frac{3}{2}}e^{-\frac{C}{\rho(t)}}.
\end{equation*}
Let us define for some constant $\widetilde{C}$
\begin{equation*}
g(t)=\widetilde{C}t^{-\frac{3}{2}}e^{-\frac{C}{l_1t}}.
\end{equation*}
We want to fine a constant $\widetilde{C}$ such that:
\begin{equation*}
\lim_{t \to 0^+}\frac{f(t)}{g(t)}=1
\end{equation*}
that is to say:
\begin{equation}\label{eq:lemconfr1}
\lim_{t \to 0^+}\frac{1}{\widetilde{C}}\dot{\rho}(t)\left(\frac{\rho(t)}{t}\right)^{-\frac{3}{2}}e^{C\left(\frac{1}{l_1t}-\frac{1}{\rho(t)}\right)}=1.
\end{equation}
To do this, let us first observe that by hypotheses R1 and R2:
\begin{equation}\label{eq:pass1zero}
l_1=\lim_{t \to 0^+}\frac{\rho(t)}{t}=\lim_{t \to 0^+}\dot{\rho}(t).
\end{equation}
Moreover we have:
\begin{equation*}
\frac{1}{l_1t}-\frac{1}{\rho(t)}=\frac{\rho(t)-l_1t}{t_1t\rho(t)}=\frac{\rho(t)-l_1t}{t^2}\frac{t}{\rho(t)}\frac{1}{l_1}
\end{equation*}
and then by hypotheses R2 and R3 we have
\begin{equation}\label{eq:pass2zero}
\lim_{t \to 0^+}\frac{1}{l_1t}-\frac{1}{\rho(t)}=\frac{l_2}{l_1^2}.
\end{equation}
Using Eqs. \eqref{eq:pass1zero} and \eqref{eq:pass2zero} in Eq. \eqref{eq:lemconfr1} we obtain
\begin{equation*}
\lim_{t \to 0^+}\frac{f(t)}{g(t)}=\frac{l_1^{-\frac{1}{2}}e^{\frac{Cl_2}{l_1^2}}}{\widetilde{C}}
\end{equation*}
and thus we have the condition
\begin{equation*}
\widetilde{C}=l_1^{-\frac{1}{2}}e^{\frac{Cl_2}{l_1^2}}>0.
\end{equation*}
Now let us define for some constants $K_1,K_2$
\begin{equation*}
H(t)=K_1t^{\frac{1}{2}}e^{-\frac{K_2}{t}}
\end{equation*}
with derivative
\begin{equation*}
h(t)=K_1e^{-\frac{K_2}{t}}\left(\frac{t^{-\frac{1}{2}}}{2}+K_2t^{-\frac{3}{2}}\right).
\end{equation*}
Let us first pose $K_2=\frac{C}{l_1}>0$ and observe that with such position we can write
\begin{equation*}
g(t)=\widetilde{C}t^{-\frac{3}{2}}e^{-\frac{K_2}{t}}.
\end{equation*}
We want to find $K_1$ such that:
\begin{equation*}
\lim_{t \to 0^+}\frac{g(t)}{h(t)}=1.
\end{equation*}
In this case we have
\begin{equation*}
1=\lim_{t \to 0^+}\frac{g(t)}{h(t)}=\lim_{t \to 0^+}\frac{\widetilde{C}}{K_1\left(\frac{t}{2}+K_2\right)}=\frac{\widetilde{C}}{K_1K_2}
\end{equation*}
and then we obtain the condition
\begin{equation*}
K_1=\frac{\widetilde{C}}{K_2}>0.
\end{equation*}
Finally let us observe that
\begin{equation*}
\lim_{t \to 0^+}F(t)=0=\lim_{t \to 0^+}H(t)
\end{equation*}
and then by using l'Hopital's rule we have
\begin{equation*}
\lim_{t \to 0^+}\frac{F(t)}{H(t)}=\lim_{t \to 0^+}\frac{f(t)}{h(t)}=\lim_{t \to 0^+}\frac{f(t)}{g(t)}\frac{g(t)}{h(t)}=1.
\end{equation*}}

\bibliographystyle{plain}

\begin{thebibliography}{10}

\bibitem{abate2000introduction}
Joseph Abate, Gagan~L Choudhury, and Ward Whitt.
\newblock An introduction to numerical transform inversion and its application
  to probability models.
\newblock In {\em Computational probability}, pages 257--323. Springer, 2000.

\bibitem{abbott1999lapicque}
Larry~F Abbott.
\newblock Lapicque’s introduction of the integrate-and-fire model neuron
  (1907).
\newblock {\em Brain research bulletin}, 50(5-6):303--304, 1999.

\bibitem{applebaum2009levy}
David Applebaum.
\newblock {\em L{\'e}vy processes and stochastic calculus}.
\newblock Cambridge university press, 2009.

\bibitem{asmussen2007stochastic}
S{\o}ren Asmussen and Peter~W Glynn.
\newblock {\em Stochastic simulation: algorithms and analysis}, volume~57.
\newblock Springer Science \& Business Media, 2007.

\bibitem{baeumer2001stochastic}
Boris Baeumer and Mark~M Meerschaert.
\newblock Stochastic solutions for fractional cauchy problems.
\newblock {\em Fractional Calculus and Applied Analysis}, 4(4):481--500, 2001.

\bibitem{bensaczuc}
E. Benedetto, L. Sacerdote and C. Zucca.
\newblock A first passage problem for a bivariate diffusion process: Numerical solution with an application to neuroscience when the process is Gauss-Markov.
\newblock {\em Journal of Computational and Applied Mathematics}, 242(1): 41 -- 52, 2013.


\bibitem{bingham1971limit}
NH~Bingham.
\newblock Limit theorems for occupation times of markov processes.
\newblock {\em Zeitschrift f{\"u}r Wahrscheinlichkeitstheorie und verwandte
  Gebiete}, 17(1):1--22, 1971.

\bibitem{bingham1989regular}
Nicholas~H Bingham, Charles~M Goldie, and Jef~L Teugels.
\newblock {\em Regular variation}, volume~27.
\newblock Cambridge university press, 1989.

\bibitem{BMHB}
Andrei N. Borodin and Paavo Salminen.
\newblock {\em Handbook of Brownian Motion - Facts and Formulae}, second edition.
\newblock Springer Basel AG, 2002.

\bibitem{buonocore2011first}
Aniello Buonocore, Luigia Caputo, Enrica Pirozzi, and Luigi~M Ricciardi.
\newblock The first passage time problem for gauss-diffusion processes:
  algorithmic approaches and applications to lif neuronal model.
\newblock {\em Methodology and Computing in Applied Probability}, 13(1):29--57,
  2011.

\bibitem{cannon2003dynamics}
Robert~H Cannon.
\newblock {\em Dynamics of physical systems}.
\newblock Courier Corporation, 2003.

\bibitem{chen2017time}
Zhen-Qing Chen.
\newblock Time fractional equations and probabilistic representation.
\newblock {\em Chaos, Solitons \& Fractals}, 102:168--174, 2017.

\bibitem{cinlar1974markov}
Erhan Cinlar et~al.
\newblock Markov additive processes and semi-regeneration.
\newblock Technical report, 1974.

\bibitem{devroye1981computer}
Luc Devroye.
\newblock On the computer generation of random variables with a given
  characteristic function.
\newblock {\em Computers \& Mathematics with Applications}, 7(6):547--552,
  1981.

\bibitem{devroye2013non}
Luc Devroye.
\newblock {\em Non-Uniform Random Variate Generation}.
\newblock Springer Science \& Business Media, 2013.

\bibitem{di2001computational}
E~Di~Nardo, AG~Nobile, E~Pirozzi, and LM~Ricciardi.
\newblock A computational approach to first-passage-time problems for
  gauss--markov processes.
\newblock {\em Advances in Applied Probability}, 33(2):453--482, 2001.

\bibitem{doob1949heuristic}
Joseph~L Doob.
\newblock Heuristic approach to the kolmogorov-smirnov theorems.
\newblock {\em The Annals of Mathematical Statistics}, pages 393--403, 1949.

\bibitem{feller1968introduction}
William Feller.
\newblock {\em An introduction to probability theory and its applications},
  volume~1.
\newblock Wiley, New York, 1968.

\bibitem{garling2014course}
David~JH Garling.
\newblock {\em A Course in Mathematical Analysis: Volume 1, Foundations and
  Elementary Real Analysis}.
\newblock Cambridge University Press, 2013.

\bibitem{gerstein1964random}
George~L Gerstein and Benoit Mandelbrot.
\newblock Random walk models for the spike activity of a single neuron.
\newblock {\em Biophysical journal}, 4(1):41--68, 1964.

\bibitem{giorno1990asymptotic}
V~Giorno, AG~Nobile, and LM~Ricciardi.
\newblock On the asymptotic behaviour of first-passage-time densities for
  one-dimensional diffusion processes and varying boundaries.
\newblock {\em Advances in applied probability}, 22(4):883--914, 1990.

\bibitem{greenwood2016stochastic}
Priscilla~E Greenwood and Lawrence~M Ward.
\newblock {\em Stochastic neuron models}, volume~1.
\newblock Springer, 2016.


\bibitem{hairer} 
M. Hairer, G. Iyer, L. Koralov, A. Novikov, and Z. Pajor-Gyulai.
A fractional kinetic process describing the intermediate time behaviour of cellular flows. \emph{The Annals of Probability}, to appear (available at arXiv:1607.01859).


\bibitem{herr}
S. Herrmann and C. Zucca.
\newblock Exact Simulation of the First-Passage Time of Diffusions.
\newblock {\em Journal of Scientific Computing,} in press.

\bibitem{hernandez2017generalised}
ME~Hern{\'a}ndez-Hern{\'a}ndez, VN~Kolokoltsov, and L~Toniazzi.
\newblock Generalised fractional evolution equations of caputo type.
\newblock {\em Chaos, Solitons \& Fractals}, 102:184--196, 2017.

\bibitem{kolokoltsov2009generalized}
Vassili~N Kolokoltsov.
\newblock Generalized continuous-time random walks, subordination by hitting
  times, and fractional dynamics.
\newblock {\em Theory of Probability \& Its Applications}, 53(4):594--609,
  2009.
  

  
  \bibitem{lansky}
P. Lansky.
\newblock On approximations of Stein's neuronal model.
\newblock {\em Journal of Theoretical Biology}, 107: 631 -- 647, 1984.
  
  
  \bibitem{Lev15} M. Levakova, M. Tamborrino, S. Ditlevsen, P. Lansky,  A review of the methods for neuronal response latency estimation.
\emph{BioSystems} 136: 23 -- 34, 2015


\bibitem{loeffen} R. Loeffen, P. Patie, M. Savov, Extinction time of non-Markovian self-similar processes, persistence, annihilation of jumps and the Fr\'echet distribution, {\em arXiv preprint arXiv:1811.07158}, 2018.

  \bibitem{maas} W. Maas.  A simple model for neural computation with firing rates and firing correlations. 
\emph{Network: Computation in Neural Systems}, 9(3): 381 -- 397, 1998.




\bibitem{magdziarz2015asymptotic}
Marcin Magdziarz and Ren{\'e} Schilling.
\newblock Asymptotic properties of brownian motion delayed by inverse
  subordinators.
\newblock {\em Proceedings of the American Mathematical Society},
  143(10):4485--4501, 2015.

\bibitem{meerschaert2008triangular}
Mark~M Meerschaert and Hans-Peter Scheffler.
\newblock Triangular array limits for continuous time random walks.
\newblock {\em Stochastic processes and their applications}, 118(9):1606--1633,
  2008.

\bibitem{meerschaert2011stochastic}
Mark~M Meerschaert and Alla Sikorskii.
\newblock {\em Stochastic models for fractional calculus}, volume~43.
\newblock Walter de Gruyter, 2011.

\bibitem{meerschaert2014semi}
Mark~M Meerschaert and Peter Straka.
\newblock Semi-markov approach to continuous time random walk limit processes.
\newblock {\em The Annals of Probability}, 42(4): 1699 -- 1723, 2014.


\bibitem{meertoa}
Mark~M Meerschaert and B. Toado.
\newblock Relaxation patterns and semi-Markov dynamics.
\newblock {\em Stochastic Processes and their Applications}, in press.


\bibitem{mehr} Mehr, C. B., and J. A. McFadden, Certain properties of Gaussian processes and their first-passage times, \textit{Journal of the Royal Statistical Society. Series B (Methodological)} (1965): 505-522.


\bibitem{Metzler}  R. Metzler and J. Klafter. The random walk's guide to anomalous diffusion: a fractional dynamics approach. \emph {Physics Reports}, 339: 1 -- 77, 2000.

\bibitem{nolan2003stable}
John Nolan.
\newblock {\em Stable distributions: models for heavy-tailed data}.
\newblock Birkhauser New York, 2003.

\bibitem{orey1968continuity}
Steven Orey et~al.
\newblock On continuity properties of infinitely divisible distribution
  functions.
\newblock {\em The Annals of Mathematical Statistics}, 39(3): 936 -- 937, 1968.
 
\bibitem{orsingher2016time}
Enzo Orsingher, Costantino Ricciuti, and Bruno Toaldo.
\newblock Time-inhomogeneous jump processes and variable order operators.
\newblock {\em Potential Analysis}, 45(3):435--461, 2016.

\bibitem{orsingher2018semi}
Enzo Orsingher, Costantino Ricciuti, and Bruno Toaldo.
\newblock On semi-markov processes and their kolmogorov's integro-differential
  equations.
\newblock {\em Journal of Functional Analysis}, 275(4):830--868, 2018.

\bibitem{sacric}
L.M. Ricciardi and L. Sacerdote.
\newblock The Ornstein-Uhlenbeck process as a model for neuronal activity.
\newblock {\em Biological Cybernetics}, 35(1): 1 -- 9, 1979.


\bibitem{rictoa}
C. Ricciuti and B. Toaldo.
\newblock Semi-Markov models and motion in heterogeneous media.
\newblock {\em Journal of Statistical Physics}, 169(2): 340 -- 361, 2017.

\bibitem{ridout2009generating}
Martin~S Ridout.
\newblock Generating random numbers from a distribution specified by its
  laplace transform.
\newblock {\em Statistics and Computing}, 19(4):439, 2009.




\bibitem{sactamzuc}
L. Sacerdote, M. Tamborrino and C. Zucca.
\newblock First passage times of two-dimensional correlated processes: Analytical results for the Wiener process and a numerical method for diffusion processes.
\newblock {\em Journal of Computational and Applied Mathematics}, 296: 275 -- 292, 2016.


\bibitem{salinas}
E. Salinas and T.J. Sejnowski
\newblock Impact of correlated synaptic input on output firing rate and variability in simple neuronal models.
\newblock {\em Journal of neuroscience}, 20(16): 6193-6209, 2000.





\bibitem{sato1999levy}
Ken-iti Sato.
\newblock {\em L{\'e}vy processes and infinitely divisible distributions}.
\newblock Cambridge university press, 1999.

\bibitem{scalas2006five}
Enrico Scalas.
\newblock Five years of continuous-time random walks in econophysics.
\newblock In {\em The complex networks of economic interactions}, pages 3--16.
  Springer, 2006.



\bibitem{toaldo2015levy}
Bruno Toaldo.
\newblock L{\'e}vy mixing related to distributed order calculus, subordinators
  and slow diffusions.
\newblock {\em Journal of Mathematical Analysis and Applications},
  430(2):1009--1036, 2015.

\bibitem{tuckwell2005introduction}
Henry~C Tuckwell.
\newblock {\em Introduction to theoretical neurobiology: volume 2, nonlinear
  and stochastic theories}, volume~8.
\newblock Cambridge University Press, 2005.

\bibitem{wuertz6rmetrics}
D~Wuertz and M~Maechler.
\newblock Rmetrics core team members (2013) stabledist: stable distribution
  functions.
\newblock {\em R package version 0.6-}.

\end{thebibliography}

\end{document}